\documentclass[reqno]{amsart}

\usepackage{amssymb}
\usepackage{hyperref}
\usepackage{mathtools}
\usepackage[all]{xy}
\usepackage{verbatim}
\usepackage{ifthen}
\usepackage{xargs}
\usepackage{bussproofs}
\usepackage{turnstile}
\usepackage{etex}

\hypersetup{colorlinks=true,linkcolor=blue}

\newcommand{\axref}[1]{(\hyperref[ax:#1]{#1})}

\newcommand{\newref}[4][]{
\ifthenelse{\equal{#1}{}}{\newtheorem{h#2}[hthm]{#4}}{\newtheorem{h#2}{#4}[#1]}
\expandafter\newcommand\csname r#2\endcsname[1]{#3~\ref{#2:##1}}
\expandafter\newcommand\csname R#2\endcsname[1]{#4~\ref{#2:##1}}
\expandafter\newcommand\csname n#2\endcsname[1]{\ref{#2:##1}}
\newenvironmentx{#2}[2][1=,2=]{
\ifthenelse{\equal{##2}{}}{\begin{h#2}}{\begin{h#2}[##2]}
\ifthenelse{\equal{##1}{}}{}{\label{#2:##1}}
}{\end{h#2}}
}

\newref[section]{thm}{Theorem}{Theorem}
\newref{lem}{Lemma}{Lemma}
\newref{prop}{Proposition}{Proposition}
\newref{cor}{Corollary}{Corollary}
\newref{cond}{Condition}{Condition}

\theoremstyle{definition}
\newref{defn}{Definition}{Definition}
\newref{example}{Example}{Example}

\theoremstyle{remark}
\newref{remark}{Remark}{Remark}

\newcommand{\bDelta}{b \Delta}

\newcommand{\fs}[1]{\mathrm{#1}}
\newcommand{\cat}[1]{\mathbf{#1}}
\newcommand{\scat}[1]{\mathcal{#1}}
\newcommand{\Fib}{\mathcal{F}}
\newcommand{\sSet}{\cat{sSet}}

\newcommand{\Hom}{\fs{Hom}}
\newcommand{\Id}{\fs{Id}}
\newcommand{\id}{\fs{id}}

\newcommand{\class}[2]{#1\text{-}\mathrm{#2}}

\newcommand{\J}{\mathrm{J}}
\newcommand{\Icof}[1][\I]{\class{#1}{cof}}
\newcommand{\Icell}[1][\I]{\class{#1}{cell}}
\newcommand{\Iinj}[1][\I]{\class{#1}{inj}}

\numberwithin{figure}{section}

\newcommand{\pb}[1][dr]{\save*!/#1-1.2pc/#1:(-1,1)@^{|-}\restore}
\newcommand{\po}[1][dr]{\save*!/#1+1.2pc/#1:(1,-1)@^{|-}\restore}

\begin{document}

\title{Contextually indexed contextual categories}

\author{Valery Isaev}

\begin{abstract}
In this paper, we define a generalization of indexed categories and contextual categories which we call contextually indexed (contextual) categories.
While contextual categories are models of ordinary type theories, contextually indexed (contextual) categories are models of indexed type theories.
We also define type-theoretic semi-fibration categories which generalize type-theoretic fibration categories.
Every model category in which cofibrations are stable under pullbacks is a type-theoretic semi-fibration category.
We show that type-theoretic semi-fibration category gives rise to a contextually indexed contextual category with finite limits.
Finally, we prove that the category of simplicial sets with the Joyal model structure gives rise to a locally small Cartesian closed contextually indexed contextual category with indexed limits and colimits.
\end{abstract}

\maketitle

\section{Introduction}

Indexed categories were defined in \cite{indexed-cats} (see also \cite[B1]{elephant}).
We defined an analogue of this notion using the language of indexed type theories which were defined in \cite{indexed-tt} as certain algebraic theories.
A contextually indexed contextual category is a model of such a theory.
We will show how to construct such a model from a category with fibrations and some additional structure.
Our construction is based on constructions from \cite{shul-inv,local-universes} which were defined for ordinary type theory.

We use the local universes construction to make constructions stable under substitution and reindexing.
To get an unstable interpretation, we use the notion of a type-theoretic semi-fibration category.
This is a category with fibrations with additional structure which allows us to interpret unit types, $\Sigma$-types, and identity types.
We do not require that all objects of a type-theoretic semi-fibration category are fibrant and we also do not require right properness.
The reason is that right properness implies that corresponding contextually indexed category is locally Cartesian closed and we want to work with models which are not.
Without right properness we cannot require all objects to be
We do not require all objects to be fibrant since we need to work with objects of the form $\Pi_p(q)$, where $p$ and $q$ are fibrations, and such an object is fibrant only if the category is right proper.

We also show how to construct extension types, products, and stable and unstable colimits.
Extension types were defined in \cite{riehl-dhott}.
They generalize $\Pi$-types and identity types.
We use them to give a description of pushouts and other higher inductive types which is based on constructions from \cite{lum-shul-hits}.
To construct unstable pushouts (and other colimits), we need to modify this argument.
We proved in \cite{indexed-tt} that the existence of pushouts in the empty context implies the existence of unstable pushouts.
A minor modification of the construction from \cite{lum-shul-hits} gives us pushouts in the empty context.

We will prove that localizations of model categories are closed under identity types.
This implies that if identity types are extensional in a model category, then they are also extensional in its localizations.
This gives us many examples of contextually indexed contextual categories with finite limits since extensionality of identity types implies the existence of equalizers as was proved in \cite{indexed-tt}.

Finally, we give one example of a type-theoretic semi-fibration category: namely, the category of simplicial sets with the Joyal model structure.
We will show that it gives rise to a locally small Cartesian closed contextually indexed contextual category with dependent limits and colimits.
To prove that it is Cartesian closed and has pushouts in the empty context, we define a functor $G : \sSet \to \sSet$.
Every edge in a simplicial set of the form $G(X)$ is invertible.
This implies that if $X$ is a quasicategory, then $G(X)$ is a Kan complex.
In general, $G(X)$ might not be a Kan complex, but it has many useful properties of Kan complexes.
For example, every categorical fibration of the form $Y \to G(X)$ is a Cartesian fibration and every map of the form $G(Y) \to G(X)$ factors into a categorical trivial cofibration followed by a Kan fibration.

The paper is organized as follows.
In section~\ref{sec:cicc}, we define contextually indexed contextual categories and show how to construct such categories from categories with fibrations.
In section~\ref{sec:ttsfc}, we define type-theoretic semi-fibration categories and prove that corresponding contextually indexed contextual categories have unit types, $\Sigma$-types, and identity types.
In section~\ref{sec:ext}, we discuss extension types.
In section~\ref{sec:colimits}, we construct dependent coproducts, pushouts, binary coproducts, and initial types.
In section~\ref{sec:loc}, we show that localizations of model categories are closed under identity types.
In section~\ref{sec:quasicats}, we discuss the contextually indexed contextual category constructed from quasicategories.

\section{Contextually indexed contextual categories}
\label{sec:cicc}

In this section, we define contextually indexed contextual categories and type-theoretic semi-fibration categories and show how to construct the former from the latter.

\subsection{Definition}

Indexed type theories are defined in \cite{indexed-tt} as certain essentially algebraic theory.
Thus, we have a notion of a model of such a theory.
We will call models of indexed unary (resp., dependent) type theories \emph{contextually indexed categories} (resp., \emph{contextually indexed contextual categories}).
We can also define them explicitly.
A contextually indexed category is a contextual category $\scat{B}$ together with a category indexed over $\scat{B}$, that is a functor $\scat{B}^\fs{op} \to \cat{Cat}$.
A contextually indexed contextual category is a contextual category $\scat{B}$ together with a contextual category indexed over $\scat{B}$, that is a functor $\scat{B}^\fs{op} \to \cat{ConCat}$.

One class of contextually indexed contextual categories was already defined in \cite{indexed-tt}.
An indexed type theory can be interpreted in an appropriate homotopy type theory.
This implies that there is a forgetful functor $U$ from the category of contextual categories to the category of contextually indexed contextual categories.
If $M$ is a contextual category, then $U(M)$ will be called \emph{the canonical indexing of $M$ over itself}.
The underlying contextual category of $U(M)$ is indeed $M$ itself.

We will use the terminology introduced in \cite{indexed-tt}, so let us recall a few relevant definitions.
A contextually indexed category $\scat{C} : \scat{B}^\fs{op} \to \cat{Cat}$ is \emph{locally small} if, for every context $\Gamma$ of $\scat{B}$ and every pair of objects $A,B \in \scat{C}(\Gamma)$,
there is a type $\Hom(A,B)$ over $\Gamma$ of maps between $A$ and $B$.
To be more precise, $\Hom(A,B)$ is a type equipped with a bijection between the set of terms of this type and the set of maps between $A$ and $B$ stable under reindexing.
Syntactically, local smallness is defined by the following rules:
\begin{center}
\AxiomC{$\Gamma \mid \cdot \vdash A$}
\AxiomC{$\Gamma \mid \cdot \vdash B$}
\BinaryInfC{$\Gamma \vdash \Hom(A,B)$}
\DisplayProof
\qquad
\AxiomC{$\Gamma \mid x : A \vdash b : B$}
\UnaryInfC{$\Gamma \vdash \lambda x.\,b : \Hom(A,B)$}
\DisplayProof
\end{center}
\medskip

\begin{center}
\AxiomC{$\Gamma \vdash f : \Hom(A,B)$}
\AxiomC{$\Gamma \mid \Delta \vdash a : A$}
\BinaryInfC{$\Gamma \mid \Delta \vdash f\,a : B$}
\DisplayProof
\end{center}

\begin{align*}
(\lambda x.\,b)\,a & = b[a/x] \\
\lambda x.\,f\,x & = f
\end{align*}

If $\scat{C}$ is a contextually indexed contextual category, then there is a stronger version of $\Hom$-types which we call \emph{dependent $\Hom$-types}.
If $B$ is an indexed type in an indexed context $\Delta$ over a base context $\Gamma$,
then a dependent $\Hom$-type $\Hom(\Delta.B)$ is a type over $\Gamma$ such that there is a bijection between the set of its terms and the set of terms of $B$ stable under reindexing.

We will say that a contextually indexed contextual category $\scat{C} : \scat{B}^\fs{op} \to \cat{ConCat}$ has some type-theoretic construction (such as $\Sigma$-types, $\Pi$-types, pushouts, and so on)
if contextual category $\scat{C}(\Gamma)$ has this construction for every context $\Gamma$ of $\scat{B}$ and this construction is stable under reindexing.
Sometimes $\scat{C}$ has some construction, but it is not stable under substitution in $\scat{C}(\Gamma)$.
In this case, we will say that $\scat{C}$ has the \emph{unstable} version of this construction.
For example, many contextually indexed contextual categories have pushouts which are not stable.
Note that we always assume that constructions are stable under reindexing, that is substitution of the base terms.

We will say that a contextually indexed contextual category has \emph{dependent products} if the following rules can be interpreted in it:
\begin{center}
\AxiomC{$\Gamma, i : I \mid \Delta \vdash B$}
\RightLabel{, $i \notin \mathrm{FV}(\Delta)$}
\UnaryInfC{$\Gamma \mid \Delta \vdash \prod_{i : I} B$}
\DisplayProof
\qquad
\AxiomC{$\Gamma, i : I \mid \Delta \vdash b : B$}
\RightLabel{, $i \notin \mathrm{FV}(\Delta)$}
\UnaryInfC{$\Gamma \mid \Delta \vdash \lambda i.\,b : \prod_{i : I} B$}
\DisplayProof
\end{center}
\medskip

\begin{center}
\AxiomC{$\Gamma \mid \Delta \vdash f : \prod_{i : I} B$}
\AxiomC{$\Gamma \vdash j : I$}
\BinaryInfC{$\Gamma \mid \Delta \vdash f\,j : B[j/i]$}
\DisplayProof
\end{center}

\begin{align*}
(\lambda i.\,b)\,j & = b[j/i] \\
\lambda i.\,f\,i & = f
\end{align*}

We will say that a contextually indexed contextual category has \emph{unstable dependent coproducts} if the following rules can be interpreted in it:
\begin{center}
\AxiomC{$\Gamma, i : I \mid \Delta \vdash B_i$}
\RightLabel{, $i \notin \mathrm{FV}(\Delta)$}
\UnaryInfC{$\Gamma \mid \Delta \vdash \coprod_{i : I} B_i$}
\DisplayProof
\qquad
\AxiomC{$\Gamma \vdash j : I$}
\UnaryInfC{$\Gamma \mid \Delta, x : B_j \vdash \fs{in}_j(x) : \coprod_{i : I} B_i$}
\DisplayProof
\end{center}
\medskip

\begin{center}
\AxiomC{$\Gamma \mid \Delta, z : \coprod_{i : I} B \vdash D$}
\AxiomC{$\Gamma, i : I \mid \Delta, x : B \vdash d : D[\fs{in}_i(x)/z]$}
\BinaryInfC{$\Gamma \mid \Delta, z : \coprod_{i : I} B \vdash \coprod\text{-}\fs{elim}(z.D, i x.d) : D$}
\DisplayProof
\end{center}
\medskip

\[ \coprod\text{-}\fs{elim}(z.D, i x.d)[\fs{in}_j(x)[b/x]/z] = d[j/i,b/x] \]

We will say that dependent coproducts are \emph{stable} if these rules are stable under substitution.
If coproducts are stable, then we can replace $\fs{in}$ with the following rule:
\begin{center}
\AxiomC{$\Gamma \vdash j : I$}
\AxiomC{$\Gamma \mid \Delta \vdash b : B_j$}
\BinaryInfC{$\Gamma \mid \Delta \vdash (j,b) : \coprod_{i : I} B_i$}
\DisplayProof
\end{center}
\medskip

\subsection{Change of base}

Let $F$ be a functor between underlying categories of contextual categories $\scat{B}'$ and $\scat{B}$.
If $\scat{C}$ is a contextual category over $\scat{B}$, then \emph{change of base} of $\scat{C}$ along $F$ is defined as $\scat{C} \circ F^\fs{op}$ and denoted by $F^*(\scat{C})$.
If $F$ and $G$ are isomorphic functors between categories underlying contextual categories $\scat{B}'$ and $\scat{B}$, then $F^*(\scat{C})$ and $G^*(\scat{C})$ are also isomorphic for every contextually indexed (contextual) category $\scat{C}$.

Let $\scat{C}$ be a category with a chosen class of maps $\Fib$, called fibrations, such that the terminal object exists and pullbacks of fibrations exist and are fibrations.
The pair $(\scat{C},\Fib)$ will be called a \emph{category with fibrations}.
We will write simply $\scat{C}$ for such a category if the class of fibrations is clear from the context.
The local universes construction defined in \cite{local-universes}, for every category with fibrations $\scat{C}$,
gives us a contextual category $\scat{C}_!$ whose types over $\Gamma$ are diagrams of the form
\[ \xymatrix{                       & E_A \ar@{->>}[d]^{p_A} \\
              \Gamma \ar[r]_-{r_A}  & V_A
            } \]
where $p_A$ is a fibration.
Terms of this type are sections $\Gamma \to E_A$.
For every type $\Gamma \vdash A$, the extended context $\Gamma, x : A$ is the pullback of this diagram.
If $\scat{C}$ satisfies additional conditions listed in \cite[Definition~4.2.1]{local-universes}, then $\scat{C}_!$ models unit types, identity types, $\Sigma$-types, and $\Pi$-types.

If $\scat{B}$ is a category with fibrations, then the underlying category of $\scat{B}_!$ is equivalent to the category $\scat{B}_\fs{f}$ of fibration objects of $\scat{B}$.
Thus, change of base for (contextual) categories indexed over $\scat{B}_!$ is defined for every functor $\scat{B}'_\fs{f} \to \scat{B}_\fs{f}$.
If $F : \scat{B}' \to \scat{B}$ is a functor that preserves fibrations, terminal objects, and pullbacks of fibrations, then there is an obvious contextual functor between contextual categories $F_! : \scat{B}'_! \to \scat{B}_!$.
The underlying functor of $F_!$ is isomorphic to $F_\fs{f} : \scat{B}'_\fs{f} \to \scat{B}_\fs{f}$, the restriction of $F$ to fibrant objects of $\scat{B}'$.

Let $\scat{B}$ be a contextual category, let $\scat{C}$ be a category with fibrations, and let $F : \scat{B} \to \scat{C}$ be a functor.
We showed above that if the image of $F$ consists of fibrant objects, then we can define change of base $F^*(\scat{C}_!)$.
Actually, we can define this contextual category for arbitrary $F$.
This is a straightforward generalization of the local universes model.
Closed indexed types in context $\Gamma$ are diagrams of the form
\[ \xymatrix{                           & E_A \ar@{->>}[d]^{p_A} \\
              F(\Gamma) \ar[r]_-{r_A}   & V_A
            } \]
where $p_A$ is a fibration.
Contexts, terms, and non-closed types are defined as before.

We are mainly interested in locally small contextually indexed categories since most of the constructions in \cite{indexed-tt} use this property.
In general, local smallness of $\scat{C}$ does not imply local smallness of $F^*(\scat{C})$.
We consider a special case when it does:

\begin{lem}[cartesian-closed]
Let $F : \scat{B} \to \scat{C}$ be a functor between categories with fibrations.
Suppose that, for every fibration $f : A \twoheadrightarrow B$ in $\scat{C}$, the pullback functor $f^* : \scat{C}/B \to \scat{C}/A$ has a right adjoint $\Pi_f : \scat{C}/A \to \scat{C}/B$.
Then
\begin{enumerate}
\item If $\Pi_f$ maps fibrations over $A$ to fibrations over $B$, then $F^*(\scat{C}_!)$ has $\Pi$-types.
\item Suppose that there is a full subcategory $\scat{C}'$ of $\scat{C}$ such that $\Pi_f$ maps fibrations over $A$ to fibrations over $B$ whenever $B$ belongs to $\scat{C}'$.
Suppose that, for every object $X \in \scat{C}$, there is a map $\varepsilon_X : X' \to X$ such that $X' \in \scat{C}'$ and,
for every fibrant object $\Gamma \in \scat{B}$, every map of the form $F(\Gamma) \to X$ factors through $\varepsilon_X$ and the factorization is natural in $\Gamma$.
Then $F^*(\scat{C}_!)$ has $\Pi$-types in the empty (indexed) context, that is $F^*(\scat{C}_!)$ is Cartesian closed.
\end{enumerate}
\end{lem}
\begin{proof}
The first claim was proved in \cite{local-universes}.
The proof of the second one is a minor modification of this proof.
A type $\Gamma \mid \cdot \vdash A$ corresponds to a map $r_A : F(\Gamma) \to V_A$ and a fibration $p_A : E_A \twoheadrightarrow V_A$.
A type $\Gamma \mid x : A \vdash B$ corresponds to a map $r_B : \Gamma.A \to V_B$ and a fibration $p_B : E_B \twoheadrightarrow V_B$, where $\Gamma.A = F(\Gamma) \times_{V_A} E_A$.
Let $V_\Pi^u = \Pi_{p_A}(E_A \times V_B)$, let $p_A^u : E_A^u \twoheadrightarrow V_\Pi^u$ be the pullback of $p_A$ along the map $V_\Pi^u \to V_A$,
and let $p_B^u : E_B^u \twoheadrightarrow E_A^u$ be the pullback of $p_B$ along the map $E_A^u \xrightarrow{\fs{ev}} E_A \times V_B \xrightarrow{\pi_2} V_B$.

The maps $r_A : F(\Gamma) \to V_A$ and $r_B : \Gamma.A \to V_B$ give us a map $r_\Pi : F(\Gamma) \to V_\Pi^u$.
By assumption, this map factors through a map $\varepsilon_{V_\Pi^u} : V_\Pi^{u'} \to V_\Pi^u$.
Let $p_A^{u'} : E_A^{u'} \twoheadrightarrow V_\Pi^{u'}$ and $p_B^{u'} : E_B^{u'} \twoheadrightarrow E_A^{u'}$ be pullbacks of $p_A^u$ and $p_B^u$ along $\varepsilon_{V_\Pi^u}$, respectively.
By assumption, $\Pi_{p_A^{u'}}(p_B^{u'})$ is a fibration over $V_\Pi^{u'}$.
We define the interpretation of $\Pi_{x : A} B$ as the map $F(\Gamma) \to V_\Pi^{u'}$ together with the fibration $\Pi_{p_A^{u'}}(p_B^{u'})$.
Since the factorization is natural in $\Gamma$, this interpretation is stable under substitution.
The interpretation of abstraction and application is defined in the same way as in the proof of the first claim.
\end{proof}

Cartesian closed contextually indexed categories are often locally small.
This is true for $F^*(\scat{C}_!)$ if $F$ has a right adjoint.
It is actually enough to assume that it has a relative right adjoint \cite{dense-relative}.
Let $\scat{B}'$ be a full subcategory.
Then a functor $G : \scat{C} \to \scat{B}$ is right adjoint to $F : \scat{B} \to \scat{C}$ relative to $\scat{B}'$ if, for every $X \in \scat{B}'$ and every $Y \in \scat{C}$,
there is a bijection $\varphi : \Hom_\scat{C}(F(X),Y) \simeq \Hom_\scat{B}(X,G(Y))$ natural in $X$ and $Y$ (actually, we only need to assume that it is natural in $X$).
Of course, any right adjoint is a right adjoint relative to any subcategory.
We will later see an example of a relative right adjoint which is not a right adjoint.

\begin{lem}[locally-small]
Let $F : \scat{B} \to \scat{C}$ be a functor between categories with fibrations which has a right adjoint $G$ relative to fibrant objects of $\scat{B}$ which preserves fibrations.
If $F^*(\scat{C}_!)$ has the $\Pi$-type $\Pi(\Delta.B)$ for some type $\Gamma \mid \Delta \vdash B$, then it has the dependent $\Hom$-type $\Hom(\Delta.B)$.
In particular, if it is Cartesian closed, then it is locally small.
\end{lem}
\begin{proof}
Let $r_\Pi : F(\Gamma) \to V_\Pi$, $p_\Pi : E_\Pi \twoheadrightarrow V_\Pi$ be the interpretation of $\Pi(\Delta.B)$.
We define $\Hom(\Delta.B)$ as $\varphi(r_\Pi) : \Gamma \to G(V_\Pi)$, $G(p_\Pi) : G(E_\Pi) \twoheadrightarrow G(V_\Pi)$.
We can apply $\varphi$ since $\Gamma$ is a context of $\scat{B}$ and contexts are always fibrant.

Since we have the abstraction and the application operations for $\Pi$-types, it is enough to define the interpretation of the following operations:
\begin{center}
\AxiomC{$\Gamma \mid \cdot \vdash b : \Pi(\Delta.B)$}
\UnaryInfC{$\Gamma \vdash \lambda(b) : \Hom(\Delta.B)$}
\DisplayProof
\qquad
\AxiomC{$\Gamma \vdash f : \Hom(\Delta.B)$}
\UnaryInfC{$\Gamma \mid \cdot \vdash f\,() : \Pi(\Delta.B)$}
\DisplayProof
\begin{align*}
\lambda(b)\,() & = b \\
\lambda(f\,()) & = f
\end{align*}
\end{center}

A term $\Gamma \mid \cdot \vdash b : \Pi(\Delta.B)$ corresponds to a section $b : F(\Gamma) \to E_\Pi$ of $p_\Pi$.
We define $\lambda(b)$ as $\varphi(b) : \Gamma \to G(E_\Pi)$.
Conversely, if $f : \Gamma \to G(E_\Pi)$ is a section of $G(p_\Pi)$, then we define $f\,()$ as $\varphi^{-1}(f) : F(\Gamma) \to E_\Pi$.
The fact that these constructions are stable under substitution follows from the naturality of $\varphi$ in $\Gamma$.
\end{proof}

\begin{remark}
If $F^*(\scat{C}_!)$ is not Cartesian closed, it might be possible to use the argument from \rlem{cartesian-closed} to prove that it is locally small directly.
If the second condition of this lemma holds except for the part that $\Pi_f$ preserves fibrations,
but it is true that $F$ has a right adjoint $G$ relative to fibrant objects such that $G(\Pi_f(g))$ is a fibration for every fibration $g$, then $F^*(\scat{C}_!)$ has dependent $\Hom$-types.
\end{remark}

Many examples of categories with fibrations that we are interested in can be constructed as localizations of larger categories.
If $\scat{M}$ is a combinatorial model category, then left Bousfield localizations of $\scat{M}$ can be described model structures on $\scat{M}$ that have the same class of cofibrations, but less fibrations.
The following proposition implies that every such localization preserves various nice properties of $\scat{M}_!$:

\begin{prop}[localized-hom]
Let $\scat{B}$ be a contextual category, let $\scat{C}$ be a category, and let $F : \scat{B} \to \scat{C}$ be a functor between them.
Let $(\scat{C},\Fib)$ and $(\scat{C},\Fib')$ be two categories with fibrations such that $\Fib' \subseteq \Fib$.
If $F^*((\scat{C},\Fib)_!)$ is locally small or has dependent $\Hom$-types, then this is also true for $F^*((\scat{C},\Fib')_!)$.
\end{prop}
\begin{proof}
Since types of $(\scat{C},\Fib')_!$ are also types of $F^*((\scat{C},\Fib)_!)$ and the base types, base terms, and indexed terms in these contextually indexed contextual categories are the same,
we can define $\Hom$-types between types of $(\scat{C},\Fib')_!$ as $\Hom$-types in $F^*((\scat{C},\Fib)_!)$.
\end{proof}

\section{Type-theoretic semi-fibration categories}
\label{sec:ttsfc}

In this section, we define type-theoretic semi-fibration categories which generalize type-theoretic fibration categories defined in \cite[Definition~2.1]{shul-inv}.

\subsection{Definition}

Let $\scat{C}$ be a category with fibrations.
Then over categories $\scat{C}/\Gamma$ are also categories with fibrations.
We will say that a map of $\scat{C}$ is a \emph{trivial cofibration} if it has the left lifting property with respect to fibrations.
Then a map in each of the categories $\scat{C}_\fs{f}$, $\scat{C}/\Gamma$, and $(\scat{C}/\Gamma)_\fs{f}$ is a trivial cofibration if and only if it is a trivial cofibration in $\scat{C}$.

Let $\scat{C}$ be a category with fibrations.
We will say that $\scat{C}$ is \emph{right proper} if, for every fibration $p : A \twoheadrightarrow B$, the right adjoint $\Pi_p : \scat{C}/A \to \scat{C}/B$ to the pullback functor $p^* : \scat{C}/B \to \scat{C}/A$ exists and maps fibrations over $A$ to fibrations over $B$.
We will say that $\scat{C}$ is \emph{cofibrantly generated} if it is cocomplete and fibrations are precisely the maps with the right lifting property with respect to a small set of maps permitting the small object argument.
We will say that $\scat{C}$ is \emph{combinatorial} if it is locally presentable and cofibrantly generated.
We will say that the class of fibrations of $\scat{C}$ is \emph{saturated} if every map that has the right lifting property with respect to trivial cofibrations is a fibration.
The class of fibrations in every cofibrantly generated category is saturated.

\begin{defn}[ttsfc]
A \emph{type-theoretic semi-fibration category} consists of a category $\scat{C}$ with a chosen class of maps, called fibrations, a chosen terminal object, and chosen pullbacks of fibrations such that the following conditions hold:
\begin{enumerate}
\item \label{it:ttsfc-pi} Fibrations are exponentiable.
\item \label{it:ttsfc-fib} The class of fibrations contains identity morphisms and is closed under compositions and pullbacks.
\item Let $i : A \to B$ be a trivial cofibration in $(\scat{C}/\Gamma)_\fs{f}$ for some $\Gamma$.
Then pullbacks of $i$ along fibrations are trivial cofibrations.
\item \label{it:ttsfc-factor} Let $i : A \to B$ be a map in $(\scat{C}/\Gamma)_\fs{f}$ for some $\Gamma$.
Then $i$ factors as a trivial cofibration followed by a fibration.
\item Let $i : A \to B$ be a trivial cofibration in $(\scat{C}/\Gamma)_\fs{f}$ for some $\Gamma$.
Then, for every map $r : \Delta \to \Gamma$, the pullback map $r^*(A) \to r^*(B)$ is a trivial cofibration.
\end{enumerate}
\end{defn}

The first condition is a technical assumption required by the local universes construction.
The second condition implies the existence of $\Sigma$-types and unit types.
The third condition is a technical assumption that guarantees that the category is homotopically well-behaved.
The last two conditions imply the existence of identity types.

The following conditions on a category with fibrations are equivalent:
\begin{itemize}
\item It is a right proper type-theoretic semi-fibration category in which all objects are fibrant.
\item It is a type-theoretic fibration category in which fibrations are exponentiable.
\end{itemize}
We put a stronger condition on our categories that fibrations are exponentiable since we need it to use the local universes construction.

\begin{remark}[semi-fib]
The results of the second and the third sections of \cite{shul-inv} hold for a type-theoretic semi-fibration category in which all objects are fibrant even if it is not right proper since this condition was not used there.
In particular, such a type-theoretic semi-fibration category is a category of fibrant objects.
\end{remark}

\begin{remark}[semi-over]
Slice categories of a type-theoretic semi-fibration category are also type-theoretic semi-fibration categories.
\end{remark}

Let $p : X \twoheadrightarrow Y$ be a fibration in a type-theoretic semi-fibration category.
We will say that $p$ is \emph{trivial} if there is a map $s : Y \to X$ such that $p \circ s = \id_Y$ and a homotopy $h : X \to P(X)$ between $s \circ p$ and $\id_X$,
where $P(X)$ is a factorization of $\langle \id_X, \id_X \rangle : X \to X \times_Y X$ into a trivial cofibration followed by a fibration.
We will say that $\scat{C}$ is \emph{extensional} if $\Pi_p$ maps trivial fibrations over $X$ to trivial fibrations over $Y$.

\begin{example}
Let $\scat{M}$ be a model category in which fibrations are exponentiable and pullbacks of trivial cofibrations are cofibrations.
Then $\scat{M}$ is a type-theoretic semi-fibration category.
We will call such a model category \emph{type-theoretic model category}.
This is a slightly more general notion than the one defined in \cite[Definition~2.12]{shul-inv}.

A type-theoretic model category is right proper as a type-theoretic semi-fibration category if and only if it is right proper as a model category.
If trivial fibrations in a type-theoretic model category $\scat{M}$ are homotopy equivalences, then $\scat{M}$ is extensional.
In particular, every right proper locally Cartesian closed model category in which cofibrations are precisely monomorphisms is an extensional type-theoretic semi-fibration category.

Every cofibrantly generated (resp., combinatorial) type-theoretic model category is cofibrantly generated (resp., combinatorial) as a type-theoretic semi-fibration category.
\end{example}

\begin{defn}
A \emph{(resp., relative) Quillen adjunction} between type-theoretic semi-fibration categories is an adjunction (resp., relative to fibrant objects) such that the right adjoint preserves fibrations.
If $F \dashv G$ is a (resp., relative) Quillen adjunction, then $F$ will be called \emph{left (resp., relative) Quillen functor} and $G$ will be called \emph{right (resp., relative) Quillen functor}.
We will say that a (relative) Quillen adjunction is \emph{extensional} if the right adjoint also preserves trivial fibrations.
\end{defn}

\begin{prop}[indexed-locally-small]
Let $F : \scat{B} \to \scat{C}$ be a functor between type-theoretic semi-fibration categories.
Then $F^*(\scat{C}_!)$ is a contextually indexed contextual category with unit types, $\Sigma$-types, and identity types.
Moreover, if $\scat{C}$ is right proper, then
\begin{enumerate}
\item $F^*(\scat{C}_!)$ has $\Pi$-types which are extensional if $\scat{C}$ is extensional.
\item If $F$ is a left relative Quillen functor, then $F^*(\scat{C}_!)$ has dependent $\Hom$-types.
In particular, it is locally small.
If $\scat{C}$ is extensional and the relative adjunction is extensional, then identity types are extensional.
\end{enumerate}
\end{prop}
\begin{proof}
The existence of unit types, $\Sigma$-types, and identity types follows from \cite{local-universes}.
The existence of $\Pi$-types and dependent $\Hom$-types is proved in \rlem{cartesian-closed} and \rlem{locally-small}.
The functional extensionality holds by \cite[Lemma~5.9]{shul-inv}.
Thus, we just need to prove that identity types are extensional.

We need to prove that the canonical function
\[ \Id_{\Hom(\Delta.B)}(f,g) \to \Hom(\Delta.\,\Id_B(f\,\overline{x},g\,\overline{x})) \] is an equivalence, where $\Delta = x_1 : A_1, \ldots x_n : A_n$.
Since we have $\Pi$-types this function is equivalent to the canonical function $\Id_{\Hom(\cdot . \Pi(\Delta,B))(\lambda(\lambda \overline{x}.f\overline{x}),\lambda(\lambda \overline {x}.g\overline{x}))} \to \Hom(\cdot . \Pi(\Delta, \Id(f\,\overline{x},g\,\overline{x})))$.
This function factors through the canonical function $\Hom(\cdot . \Id_{\Pi(\Delta,B)}(\lambda \overline{x}. f\,\overline{x}, \lambda \overline{x}. g\,\overline{x})) \to \Hom(\cdot . \Pi(\Delta, \Id(f\,\overline{x},g\,\overline{x})))$.
Since $\Pi$-types are extensional, we just need to show that the map $\Id_{\Hom(\cdot . \Pi(\Delta,B))(\lambda(\lambda \overline{x}.f\overline{x}),\lambda(\lambda \overline {x}.g\overline{x}))} \to \Hom(\cdot . \Id_{\Pi(\Delta,B)}(\lambda \overline{x}. f\,\overline{x}, \lambda \overline{x}. g\,\overline{x}))$ is an equivalence.
We prove that, more generally, for every closed type $A$ and terms $\Gamma \vdash a : \Hom(\cdot.A)$ and $\Gamma \vdash a' : \Hom(\cdot.A)$, the canonical function $\Id_{\Hom(\cdot . A)}(a,a') \to \Hom(\cdot . \Id_A(a\,(),a'\,()))$ is an equivalence.
This function will be denoted by $\fs{hap}_A$.

The type $A$ correspond to a diagram of the form
\[ \xymatrix{                           & E_A \ar@{->>}[d]^{p_A} \\
              F(\Gamma) \ar[r]_-{r_A}   & V_A
            } \]
and terms $a$ and $a'$ correspond to section $a,a' : F(\Gamma) \to E_A$ of $p_A$.
Types $\Id_{\Hom(\cdot . A)}(a,a')$ and $\Hom(\cdot . \Id_A(a\,(),a'\,()))$ are interpreted as the following diagrams:
\[ \xymatrix{                                                           & & P_{G(V_A)}(G(E_A)) \ar@{->>}[d]^q \\
              \Gamma \ar[rr]_-{\langle \varphi(a), \varphi(a') \rangle} & & G(E_A) \times_{G(V_A)} G(E_A)
            } \qquad
   \xymatrix{                                                   & & G(P_{V_A}(E_A)) \ar@{->>}[d]^{G(q')} \\
              \Gamma \ar[rr]_-{\varphi(\langle a, a' \rangle)}  & & G(E_A \times_{V_A} E_A)
            } \]
where $P_{X}(Y)$ is the path object for the diagonal $Y \to Y \times_{X} Y$.
Since the map $G(E_A) \to P_{G(V_A)}(G(E_A))$ is a trivial cofibration, we have a lift in the following diagram:
\[ \xymatrix{ G(E_A) \ar[d]_r \ar[rr]^-{G(r')}                      &                                               & G(P_{V_A}(E_A)) \ar@{->>}[d]^{G(q')} \\
              P_{G(V_A)}(G(E_A)) \ar@{->>}[r]_-q \ar@{-->}[urr]^s   & G(E_A) \times_{G(V_A)} G(E_A) \ar[r]_-\simeq  & G(E_A \times_{V_A} E_A)
            } \]
This lift is the interpretation of $\fs{hap}_A$.
To prove that the pullback of $s$ over $\Gamma$ is an equivalence which is stable under pullbacks, it is enough to show that $s$ is an equivalence.

There is a fibration $G(E_A \times_{V_A} E_A) \twoheadrightarrow G(E_A)$.
Moreover, the induced maps $X \to G(E_A)$ for every object in the diagram above are also fibrations.
By \rremark{semi-fib} and \rremark{semi-over}, we can think of this diagram as a diagram in the (non-right-proper) type-theoretic fibration category of fibrations over $G(E_A)$.
Since $r$ is trivial cofibration, \cite[Lemma~3.6]{shul-inv} implies that it is an equivalence.
The map $r'$ is also an equivalence for the same reasons.
By the 2-out-of-3 property, the map $P_{V_A}(E_A) \to E_A$ is a trivial fibration.
Since $G$ preserves trivial fibrations, the 2-out-of-3 property implies that $G(r')$ is an equivalence.
Finally, by the 2-out-of-3 property, $s$ is an equivalence.
\end{proof}

\begin{example}[cisinski-simp]
A \emph{Cisinski model structure} on a Grothendieck topos $\scat{C}$ is any cofibrantly generated model structure in which cofibrations are precisely monomorphisms.
For every such model structure and every left relative Quillen functor $F : \scat{B} \to \scat{C}$,
the contextually indexed contextual category $F^*(\scat{C}_!)$ has dependent $\Hom$-types, unit types, $\Sigma$-types, and identity types.

The existence of unit types, $\Sigma$-types, and identity types follows from \rprop{indexed-locally-small}.
By \rprop{localized-hom}, it is enough to prove that any Cisinski model structure is a left Bousfield localization of a right proper model structure.
Every such model structure is indeed a left Bousfield localization of the minimal model structure in which weak equivalences are generated by the empty set in an appropriate sense \cite[Th\'{e}or\`{e}me~3.9]{cisinski}.
The minimal model structure is right proper by \cite[Remarque~4.9]{cisinski}.
\end{example}

\subsection{Categories of functors}

Let $\scat{C}$ be a category with fibrations and let $\scat{J}$ be a small category.
A map of $\scat{C}^\scat{J}$ is a \emph{projective fibration} (resp., \emph{injective trivial cofibration}) if it is objectwise fibration (resp., objectwise trivial cofibration).
A map of $\scat{C}^\scat{J}$ is an \emph{injective fibration} if it has the right lifting property with respect to injective trivial cofibrations.
The \emph{injective} structure of a category with fibrations on $\scat{C}^\scat{J}$ has all injective fibrations as its fibrations.

\begin{lem}[inj-proj]
Let $\scat{C}$ be a cofibrantly generated category with fibrations and let $\scat{J}$ be a small category.
Then every injective fibration of $\scat{C}^\scat{J}$ is a projective fibration.
\end{lem}
\begin{proof}
Let $j$ be an object of $\scat{J}$.
Then the functor $E_j : \scat{C}^\scat{J} \to \scat{C}$ of evaluating at $j$ has a left adjoint $F : \scat{C} \to \scat{C}^\scat{J}$, which is defined as $F(X) = \Hom(j,-) \cdot X$.
If $f : A \to B$ is a trivial cofibration of $\scat{C}$, then $F(f)$ is an injective trivial cofibration.
Thus, every injective fibration $p : X \twoheadrightarrow Y$ has the right lifting property with respect to $F(f)$.
By adjointness, $p_j : X_j \to E_j$ has the right lifting property with respect to $f$.
Since $\scat{C}$ is cofibrantly generated, fibrations are precisely the maps which have the right lifting property with respect to trivial cofibrations.
Thus, $p_j$ is a fibration.
\end{proof}

\begin{prop}[injective]
Let $\scat{C}$ be a complete cofibrantly generated type-theoretic semi-fibration category and let $\scat{J}$ be a small category.
Let $\Fib$ be a class of injective fibrations closed under pullbacks, identity morphisms, and compositions such that every map in $(\scat{C}^\scat{J}/\Gamma)_\fs{f}$ for every $\Gamma$ factors into an injective trivial cofibration followed by a fibration in $\Fib$.
Then $(\scat{C}^\scat{J},\Fib)$ is a type-theoretic semi-fibration category.
\end{prop}
\begin{proof}
By \rlem{inj-proj}, fibrations in $\Fib$ are projective fibrations.
Since $\scat{C}$ is complete, fibrations in $\Fib$ are projective fibrations, and fibrations of $\scat{C}$ are exponentiable, \cite[Theorem~2.12]{comp-fact-tor} implies that fibrations of $\scat{C}^\scat{J}$ are also exponentiable.
Since fibrations in $\Fib$ are projective fibrations, to prove that other conditions hold, it is enough to prove that a map in $(\scat{C}^\scat{J}/\Gamma)_\fs{f}$ is a trivial cofibration if and only if it is an injective trivial cofibration.
The ``if'' direction holds by the assumption that fibrations are injective.
Conversely, if $i : A \to B$ is a map in $(\scat{C}^\scat{J}/\Gamma)_\fs{f}$ that has the left lifting property with respect to fibrations, then we can factor it into an injective trivial cofibration $i' : A \to A'$ followed by a fibration $A' \twoheadrightarrow B$.
Then $i$ is a retract of $i'$ which implies that it is also an injective trivial cofibration.
\end{proof}

\begin{cor}[injective-combinatorial]
Let $\scat{C}$ be a combinatorial type-theoretic semi-fibration category and let $\scat{J}$ be a small category.
Then $\scat{C}^\scat{J}$ is a combinatorial injective type-theoretic semi-fibration category.
If $\scat{C}$ is right proper, then so is $\scat{C}^\scat{J}$.
If $\scat{C}$ is a right proper type-theoretic combinatorial model category in which all object are cofibrant, then $\scat{C}^\scat{J}$ with its injective model structure is an extensional combinatorial type-theoretic model category.
\end{cor}
\begin{proof}
By \rprop{injective}, it is enough to prove that every map factors into an injective trivial cofibration followed by an injective fibration which is true by \cite[Lemma~A.2.8.3]{lurie-topos}.
If $\scat{C}$ is right proper, then the fact that $\scat{C}^\scat{J}$ is also right proper follows from the fact that injective trivial cofibrations are stable under pullbacks along fibrations which is true since injective fibrations are objectwise fibrations and pullbacks of trivial cofibrations are stable under pullbacks along fibrations in $\scat{C}$ by the right properness.
Finally, if $\scat{C}$ is a model category in which all objects are cofibrant, then objects of $\scat{C}^\scat{J}$ are also cofibrant.
It follows that trivial fibrations are precisely maps with the right lifting property with respect to cofibrations.
Thus, $\Pi$ preserves trivial fibrations whenever cofibrations are stable under pullbacks which is true since cofibrations are precisely objectwise cofibrations.
\end{proof}

\begin{example}
If $\scat{M}$ is right proper type-theoretic model category and all objects of $\scat{M}$ are cofibrant, then this is also true for $\scat{M}^\scat{J}$, which implies that it is extensional.
\end{example}

\begin{example}
If $\scat{C}$ is a combinatorial type-theoretic semi-fibration category and $F : \scat{J} \to \scat{I}$ is a functor between small categories,
then the functor $F^* : \scat{C}^\scat{I} \to \scat{C}^\scat{J}$ (defined by postcomposition with $F$) is a left Quillen functor between combinatorial type-theoretic semi-fibration categories with injective structures.
If $\scat{C}$ is also a model category in which all objects are cofibrant, then this adjunction is extensional.
We will denote the contextually indexed contextual category $(F^*)^*((\scat{C}^\scat{J})_!)$ simply by $F^*((\scat{C}^\scat{J})_!)$.
\end{example}

\subsection{Products}

Let $F$ be a functor between type-theoretic semi-fibration categories $\scat{B}$ and $\scat{C}$.
We will say that a map of $\scat{C}$ is an \emph{$F$-fibration} if it is a pullback of the map $F(f)$ for some fibration $f$.

\begin{prop}[products]
Let $F : \scat{B} \to \scat{C}$ be a functor between categories with fibrations.
Suppose that $F$ preserves pullbacks along fibrations and that, for every $F$-fibration $g : A \to B$, pullbacks of $g$ exist and the pullback functor $g^* : \scat{C}/B \to \scat{C}/A$ has a right adjoint $\Pi_g : \scat{C}/A \to \scat{C}/B$ which maps fibrations over $A$ to fibrations over $B$
Then $F^*(\scat{C}_!)$ has dependent products.
\end{prop}
\begin{proof}
First, we need to describe the interpretation of the formation rule:
\begin{center}
\AxiomC{$\Gamma, i : I \mid \Delta \vdash B$}
\RightLabel{, $i \notin \mathrm{FV}(\Delta)$}
\UnaryInfC{$\Gamma \mid \Delta \vdash \prod_{i : I} B$}
\DisplayProof
\end{center}
Suppose that we have diagrams depicted below which correspond to types $\Gamma \vdash I$ and $\Gamma, i : I \mid \Delta \vdash B$.
\[ \xymatrix{                       & E_I \ar@{->>}[d]^{p_I} \\
              \Gamma \ar[r]_{r_I}   & V_I
            } \qquad
   \xymatrix{                               & E_B \ar@{->>}[d]^{p_B} \\
              \Gamma.I.\Delta \ar[r]_{r_B}  & V_B
            } \]
where $\Gamma.I.\Delta$ is the following pullback:
\[ \xymatrix{ \Gamma.I.\Delta \ar[r] \ar@{->>}[d] \pb   & \Delta \ar@{->>}[d]^{p_\Delta} \\
              F(\Gamma.I) \ar[r] \ar[d] \pb             & F(\Gamma) \ar[d]^{F(r_I)} \\
              F(E_I) \ar[r]_{F(p_I)}                    & F(V_I)
            } \]
We define $V_\Pi$ as $\Pi_{F(p_I)}(F(E_I) \times V_B)$.
Consider the following diagrams:
\[ \xymatrix{ V_\Pi \times_{F(V_I)} F(E_I) \ar[r] \ar[d]_p \pb  & F(E_I) \ar[d]^{F(p_I)} \\
              V_\Pi \ar[r]                                      & F(V_I)
            } \qquad
   \xymatrix{ Z \ar[r] \ar@{->>}[d]_q \pb                       & F(E_I) \times E_B \ar@{->>}[d]^{\id_{F(E_I)} \times p_B} \\
              V_\Pi \times_{F(V_I)} F(E_I) \ar[r]_-{\fs{ev}}    & F(E_I) \times V_B
            } \]
We define $p_\Pi : E_\Pi \twoheadrightarrow V_\Pi$ as $\Pi_p(Z)$.
To define a map $\Delta \to V_\Pi$, it is enough to specify a map $r_I' : \Delta \to F(V_I)$ together with a map $r_B' : \Delta.I \to V_B$, where $\Delta.I$ is the pullback of $r_I'$ and $F(p_I)$.
Let $r_I' = F(r_I) \circ p_\Delta$ and $r_B' = r_B$.
These maps determine a map $[r_I',r_B] : \Delta \to V_\Pi$.
We define the interpretation of $\prod_{i : I} B$ as the following diagram:
\[ \xymatrix{                               & E_\Pi \ar@{->>}[d]^{p_\Pi} \\
              \Delta \ar[r]_{[r_I',r_B]}    & V_\Pi
            } \]

Now, let us describe the interpretation of the introduction rule:
\begin{center}
\AxiomC{$\Gamma, i : I \mid \Delta \vdash b : B$}
\RightLabel{, $i \notin \mathrm{FV}(\Delta)$}
\UnaryInfC{$\Gamma \mid \Delta \vdash \lambda i.\,b : \prod_{i : I} B$}
\DisplayProof
\end{center}
The interpretation of $b$ is a section $b : \Gamma.I.\Delta \to E_B$ of $p_B$.
Note that we have the following diagram in which the composition of bottom maps equals to $r_I'$:
\[ \xymatrix{ \Gamma.I.\Delta \ar[r]^-s \ar[d] \pb  & V_\Pi \times_{F(V_I)} F(E_I) \ar[r] \ar[d]_p \pb  & F(E_I) \ar[d]^{F(p_I)} \\
              \Delta \ar[r]_{[r_I',r_B]}            & V_\Pi \ar[r]                                      & F(V_I)
            } \]
An interpretation of $\lambda i.\,b$ is a section of $\Delta \to E_\Pi$ of $p_\Pi$.
To define such a section, it is enough to specify a section $b' : \Gamma.I.\Delta \to Z$ of $q$ over $s$.
Since $q$ is a pullback of $p_B$, this is equivalent to specifying a section of $p_B$ and we can take this section to be $b$.

Finally, we need to define the interpretation of the application:
\begin{center}
\AxiomC{$\Gamma \mid \Delta \vdash f : \prod_{i : I} B$}
\AxiomC{$\Gamma \vdash j : I$}
\BinaryInfC{$\Gamma \mid \Delta \vdash f\,j : B[j/i]$}
\DisplayProof
\end{center}
Let $f : \Delta \to E_\Pi$ be a section of $p_\Pi$ and let $j : \Gamma \to E_I$ be a section of $p_I$.
We define the interpretation of $f\,j$ as $f' \circ j''$, where $f' : \Gamma.I.\Delta \to E_B$ is a section of $p_B$ corresponding to $f$ as discussed before and $j''$ is the following pullback:
\[ \xymatrix{ \Delta \ar[r]^-{j''} \ar@{->>}[d]_{p_\Delta} \pb  & \Gamma.I.\Delta \ar[r] \ar@{->>}[d] \pb   & \Delta \ar@{->>}[d]^{p_\Delta} \\
              F(\Gamma) \ar[r]_{F(j')}                          & F(\Gamma.I) \ar[r]                        & F(\Gamma)
            } \]
where $j' : \Gamma \to \Gamma.I$ is a section of $\Gamma.I \twoheadrightarrow \Gamma$ corresponding to $j$.

It is easy to verify the stability under substitutions of the constructions that we described.
We also need to prove that $\beta$ and $\eta$ equivalences hold, but this follows from the fact that functions that we used to go from sections of $p_\Pi$ and $p_B$ and back are mutually inverse.
\end{proof}

\begin{example}[products]
Let $\scat{C}$ be a complete right proper cofibrantly generated type-theoretic semi-fibration category.
Then, for every functor $F : \scat{J} \to \scat{I}$ between small categories, the contextually indexed contextual category $F^*((\scat{C}^\scat{J})_!)$ has dependent products,
where $\scat{C}^\scat{J}$ and $\scat{C}^\scat{I}$ are equipped with the injective structures.

Indeed, $F^* : \scat{C}^\scat{I} \to \scat{C}^\scat{J}$ preserves all limits.
By \rlem{inj-proj}, $F^*$-fibrations are objectwise fibrations.
Since fibrations are exponentiable in $\scat{C}$, \cite[Theorem~2.12]{comp-fact-tor} and \cite[Corollary~2.6]{comp-fact-tor} imply that $F^*$-fibrations are also exponentiable in $\scat{C}^\scat{J}$.
The fact that $\Pi_g$ preserves fibrations is equivalent to the fact that injective trivial cofibrations are stable under pullbacks along $F^*$-fibrations.
This follows from the facts that $\scat{C}$ is right proper and that $F^*$-fibrations are objectwise fibrations.
\end{example}

\section{Extension types}
\label{sec:ext}

Extension types were defined in \cite{riehl-dhott}.
In this section, we describe an analogous construction in ordinary homotopy type theory and in indexed type theory.

Extension types generalize $\Pi$-types (and products in the indexed case).
Semantically, if $j : U \to V$ is a map, $p : A \twoheadrightarrow V$ is a fibration, and $a : U \to A$ is a section of $p$ over $j$,
then the extension object $\langle \Pi_V(p)\!\!\mid^j_a \rangle$ is the object of sections of $p$ which restrict to $a$.
It is easy to describe their syntax in ordinary type theory:
\begin{center}
\AxiomC{$\Gamma, y : V \vdash A$}
\AxiomC{$\Gamma, x : U \vdash a : A[j\,x/y]$}
\BinaryInfC{$\Gamma \vdash \langle \Pi_{y : V} A \mid_{x.a} \rangle$}
\DisplayProof
\end{center}
\smallskip

\begin{center}
\AxiomC{$\Gamma, y : V \vdash a : A$}
\UnaryInfC{$\Gamma \vdash \lambda y.a : \langle \Pi_{y : V} A \mid_{x.a[ix/y]} \rangle$}
\DisplayProof
\qquad
\AxiomC{$\Gamma \vdash f : \langle \Pi_{y : V} A \mid_{x.a} \rangle$}
\AxiomC{$\Gamma \vdash v : V$}
\BinaryInfC{$\Gamma \vdash \fs{app}_{x.a}(f,v) : A[v/y]$}
\DisplayProof
\begin{align*}
\fs{app}_{x.a'}(\lambda y.a,v) & = a[v/y] \\
\lambda y.\,\fs{app}_{x.a}(f,y) & = f \\
\fs{app}_{x.a}(f,j\,u) & = a[u/x]
\end{align*}
\end{center}

The problem is that the object $U$ is often not representable in the syntax.
Of course, we might add it explicitly to the syntax, but it might be impossible to capture its universal property strictly.
This problem can be solved if $U$ is a finite colimit in which all objects and arrows are representable and equality between arrows holds strictly in the syntax.
Let $U$ be such a colimit $\fs{colim}_{i \in \scat{I}} U_i$ and suppose that, for every $i \in \scat{I}$, the composite $U_i \to \fs{colim}_{i \in \scat{I}} \xrightarrow{j} V$, denoted by $j_i$, is also representable in the syntax.
Then the extension type is written as $\langle \Pi_{y : V} A \mid_{x_1.a_1, \ldots x_n.a_n} \rangle$, where $n$ is the number of objects of $\scat{I}$.
Similarly, $\fs{app}_{x.a}(f,v)$ is replaced with $\fs{app}_{x_1.a_1, \ldots x_n.a_n}(f,v)$.
The single premise $\Gamma, x : U \vdash a : A[j\,x/y]$ in the first rule is replaced with premises $\Gamma, x_i : U_i \vdash a_i : A[j_i\,x_i/y]$ for each $i \in \scat{I}$.
We also put the condition that terms $a_i$ agree with each other, that is, for every map $f : i \to i'$ of $\scat{I}$, we have a strict equality in the premise $\Gamma, x_i : U_i \vdash a_i \equiv a_{i'}[f\,x_i/x_{i'}] : A[j_i\,x_i/y]$.
The final equation is replaced with equations $\fs{app}_{x_1.a_1, \ldots x_n.a_n}(f,j_i\,u_i) = a_i[u_i/x_i]$ for each $i \in \scat{I}$.
Other rules and equations do not change.

One example of an extension type is the type of extensions along a map of the form $1 \amalg 1 \to \Delta^1$, where $1$ is the terminal object and $\Delta^1$ is some object.
The type of extensions in $A$ along such a map is the type of morphisms in the simplicial space (or $\infty$-category) $A$ in a version of the type theory described in \cite{riehl-dhott} if $\Delta^1$ is interpreted appropriately.
On the other hand, we can interpret $\Delta^1$ as a contractible type and add a rule to the syntax which implies that it is contractible.
Then extension types along $1 \amalg 1 \to \Delta^1$ are equivalent to (heterogeneous) identity types.
Such extension types were defined and studied in \cite[Subsection~3.2]{alg-models}.
We will call them \emph{path types} since they are literally types of paths (that is, maps from the interval) between two specified points.

Path types can be used to give a convenient description of higher inductive types.
For example, let $f : A \to B$ and $g : A \to C$ be a pair of maps.
Then their pushout $B \amalg_A C$ has three constructors: $\fs{inl} : B \to B \amalg_A C$, $\fs{inr} : C \to B \amalg_A C$, and $\fs{glue} : A \to \Delta^1 \to B \amalg_A C$.
The last constructor satisfies equations $\fs{glue}\,a\,j_0 = f\,a$ and $\fs{glue}\,a\,j_1 = g\,a$, where $[j_0,j_1] = j : 1 \amalg 1 \to \Delta^1$.
The eliminator is also easy to describe:
given a dependent type $D$ over $w : B \amalg_A C$ and functions $d_1 : \Pi_{y : B} D[\fs{inl}\,y/w]$, $d_2 : \Pi_{z : C} D[\fs{inr}\,z/w]$,
and $d_3 : \Pi_{x : A} \Pi_{i : \Delta^1} D[\fs{glue}\,x\,i/w]$ satisfying equations $d_3\,a\,j_0 = d_1\,(f\,a)$ and $d_3\,a\,j_1 = d_2\,(g\,a)$,
there is a section of $D$.
It is easy to show that this description of pushouts is equivalent to the standard one.

Now, after this digression, let us go back to the discussion of extension types.
Of course, we can use them in the indexed type theory.
We can add them to the base theory and to the indexed theory, but there is a third option.
We can add the type of extension in an indexed type along a base map.
If ordinary extension types generalize $\Pi$-types, then these extension types generalize products.
It should be clear how to modify the rules to get such extension types.
Let us just give the first rule as an example (we use the version with the single type $U$ to simplify the notation):
\begin{center}
\AxiomC{$\Gamma, y : V \mid \Delta \vdash A$}
\AxiomC{$\Gamma, x : U \mid \Delta \vdash a : A[j\,x/y]$}
\RightLabel{, $x,y \notin \mathrm{FV}(\Delta)$}
\BinaryInfC{$\Gamma \mid \Delta \vdash \langle \prod_{y : V} A \mid_{x.a} \rangle$}
\DisplayProof
\end{center}

To describe the interpretation of extension types, we need to introduce a new notion.
We will say that a category with fibrations satisfies \emph{the pushout-product axiom} with respect to a map $j : U \to V$ if $U$ and $V$ are exponentiable and,
for every fibration $X \twoheadrightarrow Y$, the map $X^V \to Y^V \times_{Y^U} X^U$ is a fibration.

\begin{remark}
Suppose that the class of fibrations of $\scat{C}$ is saturated and objects $U$ and $V$ are exponentiable.
Then $\scat{C}$ satisfies the pushout-product axiom with respect to a map $j : U \to V$ if and only if,
for every trivial cofibration $i : A \to B$, the map $B \times U \amalg_{A \times U} A \times V \to B \times V$ is also a trivial cofibration.
This condition holds if $\scat{C}$ is a Cartesian closed model category and $j$ is a cofibration.
\end{remark}

\begin{prop}[extension-types]
If a category with fibrations $\scat{C}$ satisfies the pushout-product axiom with respect to a map $j : U \to V$, then it has extension types along $j$.
If moreover $U = \fs{colim}_{i \in \scat{I}} U_i$ and the functor $X \times -$ preserves this colimit for every $X$, then $\scat{C}$ has the modified version of extension types based on this decomposition of $U$.
\end{prop}
\begin{proof}
The premises of the introduction rule for $\langle \Pi_{y : V} A \mid_{x.a} \rangle$ give us the following commutative square:
\[ \xymatrix{ \Gamma \times U \ar[r]^a \ar[d]_{\id \times j}    & E_A \ar@{->>}[d]^{p_A} \\
              \Gamma \times V \ar[r]_{r_A}                      & V_A
            } \]
The extension type $\langle \Pi_{y : V} A \mid_{x.a} \rangle$ must be interpreted as the object of lifts in this square.
Thus, we define $p_\Pi : E_\Pi \twoheadrightarrow V_\Pi$ as the map $E_A^V \to V_A^V \times_{V_A^U} E_A^U$ which is a fibration by assumption.
Maps $a : \Gamma \times U \to E_A$ and $r_A : \Gamma \times V \to V_A$ determine a map $r_\Pi : \Gamma \to V_\Pi$.
This gives us the interpretation of the extension type.
It is easy to describe the interpretation of abstraction and application.

If $U = \fs{colim}_{i \in \scat{I}} U_i$, then we have a collection of maps $a_i : \Gamma \times U_i \to E_A$ for every $i \in \mathcal{I}$
which determine a map $\fs{colim}_{i \in \mathcal{I}}(\Gamma \times U_i) \to E_A$.
Since $\Gamma \times -$ preserves the colimit $\fs{colim}_{i \in \mathcal{I}} U_i$, this gives us a map $\Gamma \times U \to E_A$.
Now, we can use the previous construction.
\end{proof}

\begin{prop}[indexed-extension-types]
Let $F : \scat{B} \to \scat{C}$ be a functor between categories with fibrations which preserves binary products.
Let $j : U \to V$ be a map of $\scat{B}$.
If $\scat{C}$ satisfies the pushout-product axiom with respect to the map $F(j)$, then it has indexed extension types along $j$.
\end{prop}
\begin{proof}
The proof is the same as the proof of the previous proposition.
\end{proof}

\begin{remark}
The interpretation of the indexed version of extension types is defined in the same way.
We only need to know that $F$ preserves products and that $\scat{C}$ satisfies the pushout-product axiom with respect to the map $F(j) : F(U) \to F(V)$.
\end{remark}

\section{Colimits}
\label{sec:colimits}

In this section, we give sufficient conditions for a model $F^*(\scat{C})$ to have initial types, binary coproducts, arbitrary coproducts, and pushouts.

\subsection{Quasifibrations}

For the discussion of colimits, it will be convenient to use the notion of a quasifibration which we now define.
We will say that a map $f : X \to Z$ in a category with fibrations is a \emph{quasifibration} if it factors as a trivial cofibration $i : X \to Y$ followed by a fibration $p : Y \twoheadrightarrow Z$ and, for every diagram of the form
\[ \xymatrix{ X' \ar[r] \ar[d]_{i'} \pb         & X \ar[d]^i \\
              Y' \ar[r] \ar@{->>}[d]_{p'} \pb   & Y \ar@{->>}[d]^p \\
              Z' \ar[r]                         & Z
            } \]
the map $i'$ is a trivial cofibration.
This notion is similar to the classical notion of a quasifibration between spaces.
Let $\scat{C}'$ be a full subcategory of $\scat{C}$ such that $Z$ belongs to $\scat{C}'$.
We will say that $f : X \to Z$ is a \emph{quasifibration relative to $\scat{C}'$} if the condition above holds whenever $Z'$ belongs to $\scat{C}'$.

\begin{example}
Every fibration is a quasifibration.
\end{example}

\begin{lem}[proj-quasi]
Let $\scat{C}$ be a combinatorial type-theoretic semi-fibration category and let $\scat{J}$ be a small category.
Then every projective fibration of $\scat{C}^\scat{J}$ is a quasifibration in the injective structure.
\end{lem}
\begin{proof}
Let $f : X \to Z$ be a projective fibration.
The factorization of $f$ into an injective trivial cofibration $i : X \to Y$ followed by an injective fibration $p : Y \twoheadrightarrow Z$ exists by \cite[Lemma~A.2.8.3]{lurie-topos}.
By \rlem{inj-proj}, injective fibrations are projective.
Thus, for every $j \in \scat{J}$, the map $i_j : X_j \to Y_j$ is a trivial cofibration between fibrant objects in $\scat{C}/Z$.
It follows that its pullbacks along maps $Z' \to Z$ are trivial cofibrations by the definition of a type-theoretic semi-fibration category.
\end{proof}

The following lemma was proved in \cite[Theorem~3.3]{lum-shul-hits}.
We just extract the minimal set of conditions that was used in the proof.

\begin{lem}[quasifib]
In a right proper type-theoretic model category, stable under pullbacks coproducts of fibrations over a fixed base are quasifibrations.
\end{lem}

Finally, let us prove an easy but useful lemma:

\begin{lem}
Let $\scat{C}$ be a type-theoretic model category and let $\scat{C}'$ be a full subcategory of $\scat{C}$.
Suppose that, for every fibration $p : X \twoheadrightarrow Z$ with $Z \in \scat{C}'$, every trivial cofibration $U \to V$, and every map $V \to Z$, the induced map $U \times_Z X \to V \times_Z X$ is a trivial cofibration.
Let $f : X \to Z$ be a map such that $Z \in \scat{C}'$ and, for every trivial cofibration $U \to V$ with $U \in \scat{C}'$ and every map $V \to Z$, the induced map $U \times_Z X \to V \times_Z X$ is a trivial cofibration.
Then $f$ is a quasifibration relative to $\scat{C}'$.
\end{lem}
\begin{proof}
Factor $f$ it into a trivial cofibration $X \to Y$ followed by a fibration $Y \twoheadrightarrow Z$.
Let $Z' \to Z$ be a map such that $Z' \in \scat{C}'$.
Factor $Z' \to Z$ into a trivial cofibration $Z' \to Z''$ followed by a fibration $Z'' \twoheadrightarrow Z$:
\[ \xymatrix{ X' \ar[r] \ar[d] \pb          & X'' \ar@{->>}[r] \ar[d] \pb       & X \ar[d] \\
              Y' \ar[r] \ar@{->>}[d] \pb    & Y'' \ar@{->>}[r] \ar@{->>}[d] \pb & Y \ar@{->>}[d] \\
              Z' \ar[r]                     & Z'' \ar@{->>}[r]                  & Z
            } \]
The maps $X'' \to Y''$ and $Y' \to Y''$ are trivial cofibrations by assumption.
The map $X' \to X''$ is also a trivial cofibration by assumption on $f$.
Thus, by the 2-out-of-3 property, the map $X' \to Y'$ is also a trivial cofibration.
\end{proof}

\subsection{Coproducts}

In this subsection, we describe the interpretation of the initial type, binary coproducts, and arbitrary coproducts.

\begin{prop}[initial]
Let $\scat{B}$ be a contextual category, let $\scat{C}$ be a category with fibrations, and let $F : \scat{B} \to \scat{C}$ be a functor between them.
If $\scat{C}$ has a strict quasifibrant initial object, then $F^*(\scat{C})$ has strict initial types.
\end{prop}
\begin{proof}
Factor the map $0 \to 1$ into a trivial cofibration $0 \to R(0)$ followed by a fibration $R(0) \twoheadrightarrow 1$.
We define the initial type in any context as follows:
\[ \xymatrix{               & R(0) \ar@{->>}[d] \\
              \Delta \ar[r] & 1
            } \]
By \cite[Proposition~7.4]{indexed-tt}, we just need to described the interpretation of the elimination rule:
\begin{center}
\AxiomC{$\Gamma \mid \Delta \vdash D$}
\AxiomC{$\Gamma \mid \Delta \vdash a : 0$}
\BinaryInfC{$\Gamma \mid \Delta \vdash 0\text{-}\fs{elim'}(D,a) : D$}
\DisplayProof
\end{center}

Let $a : \Delta \to R(0)$ be a map in $\scat{C}$ and let $r_D : \Delta \to V_D$, $p_D : E_D \twoheadrightarrow V_D$ be the interpretation of $D$.
Then we need to construct a section of $p_D$ over $r_D$.
The map $0 \times V_D \to R(0) \times V_D$ is a trivial cofibration since it is a pullback of $0 \to R(0)$ and $0$ is quasifibrant.
Since $0$ is strict, $0 \times V_D$ is initial.
Hence, the unique map $0 \to R(0) \times V_D$ is a trivial cofibration.
It follows that we have a lift in the following diagram:
\[ \xymatrix{ 0 \ar[r] \ar[d]                                   & E_D \ar@{->>}[d]^{p_D} \\
              V_D \times R(0) \ar[r]_-{\pi_1} \ar@{-->}[ur]^s   & V_D
            } \]
Then the interpretation of $0\text{-}\fs{elim'}$ is defined as $\Delta \xrightarrow{\langle r_D, a \rangle} V_D \times R(0) \xrightarrow{s} E_D$.
\end{proof}

\begin{remark}
Let $\scat{C}$ be a category with fibrations with a strict initial object $0$.
If $0$ is fibrant, then trivial cofibrations with the initial domain also have the initial codomain.
The converse holds if the class of fibrations is saturated.
\end{remark}

\begin{example}
Let $\scat{C}$ be a category with fibrations and let $\scat{J}$ be a small category.
If $\scat{C}$ has a strict fibrant initial object $0$, then $\scat{C}^\scat{J}$ has a strict injectively fibrant initial object.
Indeed, the constant functor on $0$ is a strict initial object.
The previous remark implies that it is injectively fibrant.
\end{example}

\begin{example}
If $\scat{C}$ is a right proper type-theoretic model category with a strict initial object, then $\scat{C}^\scat{J}$ has a strict injectively quasifibrant initial object.
This follows from \rlem{quasifib} and the fact that $\scat{C}^\scat{J}$ is right proper.
\end{example}

Since the definition of binary coproducts in indexed type theories does not involve the base context, we can use the same construction as in ordinary type theory.
Such a construction was described in \cite{lum-shul-hits}.
The only difference between contextual categories and indexed contextual categories is that not all context are fibrant in the indexed case;
it is only true that context are fibrant over $F(\Gamma)$, where $\Gamma$ is a base context, and $F(\Gamma)$ is usually not fibrant.
This condition was not used in \cite{lum-shul-hits}, so we can apply theorems from this paper.

\begin{prop}[binary-coproducts]
Let $\scat{B}$ be a contextual category, let $\scat{C}$ be a category with fibrations, and let $F : \scat{B} \to \scat{C}$ be a functor between them.
Suppose that, for every pair of fibrations $A \twoheadrightarrow \Gamma$ and $B \twoheadrightarrow \Gamma$ of $\scat{C}$,
there exists the coproduct $A \amalg B$ and the induced map $A \amalg B \to \Gamma$ is a quasifibration.
Then $F^*(\scat{C}_!)$ has stable dependent coproducts.
\end{prop}
\begin{proof}
The construction is essentially the same as the construction of coproducts in \cite[Theorem~3.3]{lum-shul-hits}.
\end{proof}

\begin{example}
Let $\scat{C}$ be a category with fibrations.
Suppose that there is a class $S$ of maps of $\scat{C}$ such that the following conditions hold:
\begin{enumerate}
\item A map is a fibration if and only if it has the right lifting property with respect to $S$.
\item Binary coproducts exist in $\scat{C}$ and the domains of maps in $S$ are connected in the sense that maps from them to a coproduct $A \amalg B$ factor through either $A$ or $B$.
\end{enumerate}
Then fibrations over $\Gamma$ are closed under coproducts.
Thus, \rprop{binary-coproducts} implies that $F^*(\scat{C}_!)$ has stable dependent coproducts for every functor $F : \scat{B} \to \scat{C}$.
\end{example}

Finally, let us construct indexed coproducts:

\begin{prop}[coproducts]
Let $F : \scat{B} \to \scat{C}$ be a functor between categories with fibrations.
Suppose that the following conditions hold:
\begin{enumerate}
\item $F$ preserves pullbacks along fibrations and pullbacks of $F$-fibrations exist.
\item Fibrations and $F$-fibrations are exponentiable.
\item The composition of a fibration and an $F$-fibration is a quasifibration.
\end{enumerate}
Then $F^*(\scat{C}_!)$ has stable dependent coproducts.
\end{prop}
\begin{proof}
Suppose that we have diagrams depicted below which correspond to types $\Gamma \vdash I$ and $\Gamma, i : I \mid \Delta \vdash B$.
\[ \xymatrix{                       & E_I \ar@{->>}[d]^{p_I} \\
              \Gamma \ar[r]_{r_I}   & V_I
            } \qquad
   \xymatrix{                               & E_B \ar@{->>}[d]^{p_B} \\
              \Gamma.I.\Delta \ar[r]_{r_B}  & V_B
            } \]
where $\Gamma.I.\Delta$ is the following pullback:
\[ \xymatrix{ \Gamma.I.\Delta \ar[r] \ar@{->>}[d] \pb   & \Delta \ar@{->>}[d]^{p_\Delta} \\
              F(\Gamma.I) \ar[r] \ar[d] \pb             & F(\Gamma) \ar[d]^{F(r_I)} \\
              F(E_I) \ar[r]_{F(p_I)}                    & F(V_I)
            } \]
We define $V_\amalg$ as $\Pi_{F(p_I)}(F(E_I) \times V_B)$.
Consider the following diagrams:
\[ \xymatrix{ V_\amalg \times_{F(V_I)} F(E_I) \ar[r] \ar[d]_p \pb   & F(E_I) \ar[d]^{F(p_I)} \\
              V_\amalg \ar[r]                                       & F(V_I)
            } \qquad
   \xymatrix{ Z \ar[r] \ar@{->>}[d]_q \pb                       & F(E_I) \times E_B \ar@{->>}[d]^{\id_{F(E_I)} \times p_B} \\
              V_\amalg \times_{F(V_I)} F(E_I) \ar[r]_-{\fs{ev}} & F(E_I) \times V_B
            } \]
Since $p \circ q$ is a quasifibration, we can factor it into a trivial cofibration $t : Z \to E_\amalg$ followed by a fibration $p_\amalg : E_\amalg \twoheadrightarrow V_\amalg$.
To define a map $\Delta \to V_\amalg$, it is enough to specify a map $r_I' : \Delta \to F(V_I)$ together with a map $r_B' : \Delta.I \to V_B$, where $\Delta.I$ is the pullback of $r_I'$ and $F(p_I)$.
Let $r_I' = F(r_I) \circ p_\Delta$ and $r_B' = r_B$.
These maps determine a map $r_\amalg = [r_I',r_B] : \Delta \to V_\amalg$.
We define the interpretation of $\coprod_{i : I} B$ as the pair $r_\amalg$, $p_\amalg$.

Now, let us describe the interpretation of the introduction rule:
\begin{center}
\AxiomC{$\Gamma \vdash j : I$}
\AxiomC{$\Gamma \mid \Delta \vdash b : B_j$}
\BinaryInfC{$\Gamma \mid \Delta \vdash (j,b) : \coprod_{i : I} B_i$}
\DisplayProof
\end{center}
The interpretation of $j$ is a section $j : \Gamma \to E_I$ of $p_I$ over $r_I$.
This map determines a section $j' : \Gamma \to \Gamma.I$ of the map $\Gamma.I \to \Gamma$.
The pullback of $F(j')$ along $p_\Delta$ will be denoted by $j'' : \Delta \to \Gamma.I.\Delta$.
The interpretation of $b$ is a section $b : \Gamma.I.\Delta \to E_B$ of $p_B$.
This map determines a section $b' : \Gamma.I.\Delta \to Z$ of $q$:
\[ \xymatrix{                                                   & Z \ar[r] \ar@{->>}[d]^q \pb                                       & F(E_I) \times E_B \ar@{->>}[d]^{\id_{F(E_I) \times p_B}}  &                           \\
              \Gamma.I.\Delta \ar[r] \ar[d] \ar[ur]^{b'} \pb    & V_\amalg \times_{F(V_I)} F(E_I) \ar[r]_-{\fs{ev}} \ar[d]^p \pb    & F(E_I) \times V_B \ar[r]_-{\pi_1}                         & F(E_I) \ar[d]^{F(p_I)}    \\
              \Delta \ar[r]_{[r_I',r_B]}                        & V_\amalg \ar[rr]                                                  &                                                           & F(V_I)
            } \]
We define the interpretation of $(j,b)$ as $t \circ b' \circ j''$.

Finally, let us describe the interpretation of the elimination rule:
\begin{center}
\def\extraVskip{1pt}
\Axiom$\fCenter \Gamma \mid \Delta, z : \coprod_{i : I} B_i \vdash D$
\noLine
\UnaryInf$\fCenter \Gamma, i : I \mid \Delta, x : B_i \vdash d : D[(i,x)/z]$
\Axiom$\fCenter \Gamma \mid \Delta \vdash c : \coprod_{i : I} B_i$
\def\extraVskip{2pt}
\BinaryInfC{$\Gamma \mid \Delta \vdash \coprod\text{-}\fs{elim}(z.D, i x.d, c) : D$}
\DisplayProof
\end{center}
The interpretation of $D$ consists of a map $r_D : \Delta.\amalg \to V_D$ together with a fibration $p_D : E_D \twoheadrightarrow V_D$, where $\Delta.\amalg$ is the pullback of $r_\amalg$ and $p_\amalg$.
The interpretation of $d$ is a section $d : \Gamma.I.\Delta.B \to E_D$ of $p_D$, where $\Gamma.I.\Delta.B$ is the pullback of $r_B$ and $p_B$.
The interpretation of $c$ is a section $c : \Delta \to E_\amalg$ of $p_\amalg$ over $r_\amalg$.

Let $s$ be the pullback of $p \circ q$:
\[ \xymatrix{ \Pi_{p_\amalg}(E_\amalg \times V_D) \times_{V_\amalg} Z \ar[r] \ar[d]_s \pb   & Z \ar[d]^{p \circ q} \\
              \Pi_{p_\amalg}(E_\amalg \times V_D) \ar[r]                                    & V_\amalg
            } \]
Let $X$ be the following pullback:
\[ \xymatrix{ X \ar[rr] \ar@{->>}[d] \pb                                                            &                                                                                   & E_\amalg \times E_D \ar@{->>}[d]^{\id \times p_D} \\
              \Pi_{p_\amalg}(E_\amalg \times V_D) \times_{V_\amalg} Z \ar[r]_-{\id \times_\id t}    & \Pi_{p_\amalg}(E_\amalg \times V_D) \times_{V_\amalg} E_\amalg \ar[r]_-{\fs{ev}}  & E_\amalg \times V_D
            } \]
Maps $p$ and $q$ are exponentiable by assumption.
By \cite[Corollary~2.6]{comp-fact-tor}, $s$ is also exponentiable.
Let $t'$ be the pullback of $t$:
\[ \xymatrix{ \Pi_s(X) \times_{V_\amalg} Z \ar[r] \ar[d]_{t'} \pb           & Z \ar[d]^t                        \\
              \Pi_s(X) \times_{V_\amalg} E_\amalg \ar[r] \ar@{->>}[d] \pb   & E_\amalg \ar@{->>}[d]^{p_\amalg}  \\
              \Pi_s(X) \ar[r]                                               & V_\amalg
            } \]
Since $p \circ q$ is a quasifibration, $t'$ is a trivial cofibration.
It follows that we have a lift in the following diagram:
\[ \xymatrix{ \Pi_s(X) \times_{V_\amalg} Z \ar[r]^-{\fs{ev}} \ar[d]_{t'}    & X \ar[rr]                                                                         &                                       & E_D \ar@{->>}[d]^{p_D} \\
              \Pi_s(X) \times_{V_\amalg} E_\amalg \ar[r] \ar@{-->}[urrr]^e  & \Pi_{p_\amalg}(E_\amalg \times V_D) \times_{V_\amalg} E_\amalg \ar[r]_-{\fs{ev}}  & E_\amalg \times V_D \ar[r]_-{\pi_2}   & V_D
            } \]
We define the interpretation of the eliminator as the composite $\Delta \to \Pi_s(X) \times_{V_\amalg} E_\amalg \xrightarrow{e} E_D$, where the first map is defined below.
To define such a map, it is enough to specify maps $\Delta \to \Pi_s(X)$ and $\Delta \to E_\amalg$ which are equal over $V_\amalg$.
We define the latter map as $c : \Delta \to E_\amalg$.
To define the former map, it is enough to specify two maps:
\begin{itemize}
\item A map $\Delta \to \Pi_{p_\amalg}(E_\amalg \times V_D)$ such that its composition with the map $\Pi_{p_\amalg}(E_\amalg \times V_D) \to V_\amalg$ equals to $r_\amalg$.
To define such a map, it is enough to specify a map $\Delta.\amalg \to V_D$.
We define it as $r_D$.
\item A section $\Gamma.I.\Delta.B \to E_D$ of $p_D$.
We define this section as $d$.
\end{itemize}
This completes the definition of the interpretation of the eliminator.
\end{proof}

\begin{example}
If $\scat{C}$ is a combinatorial type-theoretic semi-fibration category and $F : \scat{J} \to \scat{I}$ is a functor between small categories
then $F^*((\scat{C}^\scat{J})_!)$ satisfies the conditions of \rprop{coproducts}, where $\scat{C}^\scat{J}$ is equipped with the injective structure.
Indeed, both injective fibrations and $F$-fibrations are objectwise fibrations which implies that they are exponentiable.
Moreover, \rlem{proj-quasi} implies that their composition is a quasifibration.
\end{example}

The following proposition gives sufficient conditions for unstable coproducts to exist:

\begin{prop}[unstable-coproducts]
Let $F : \scat{B} \to \scat{C}$ be a functor between categories with fibrations.
Suppose that the following conditions hold:
\begin{enumerate}
\item $F$ preserves pullbacks along fibrations and pullbacks of $F$-fibrations exist.
\item Fibrations and $F$-fibrations are exponentiable.
\item The composition of a fibration and an $F$-fibration is a quasifibration relative to a full subcategory $\scat{C}'$ of $\scat{C}$.
\item For every object $X \in \scat{C}$, there is a map $\varepsilon_X : X' \to X$ such that $X' \in \scat{C}'$ and,
for every fibrant object $\Gamma \in \scat{B}$, every map of the form $F(\Gamma) \to X$ factors through $\varepsilon_X$ and the factorization is natural in $\Gamma$.
\end{enumerate}
Then $F^*(\scat{C}_!)$ has coproducts.
\end{prop}
\begin{proof}
We have as before the interpretation of types $\Gamma \vdash I$ and $\Gamma, i : I \mid \cdot \vdash B$:
\[ \xymatrix{                       & E_I \ar@{->>}[d]^{p_I} \\
              \Gamma \ar[r]_{r_I}   & V_I
            } \qquad
   \xymatrix{ F(\Gamma.I) \ar[r] \ar[d] \pb & F(E_I) \ar[d]^{F(p_I)} \\
              F(\Gamma) \ar[r]_{F(r_I)}     & F(V_I)
            } \qquad
   \xymatrix{                           & E_B \ar@{->>}[d]^{p_B} \\
              F(\Gamma.I) \ar[r]_{r_B}  & V_B
            } \]
Let $V_\amalg = \Pi_{F(p_I)}(F(E_I) \times V_B)$ as before.
We define $p : V_\amalg \times_{F(V_I)} F(E_I) \to V_\amalg$ and $q : Z \to V_\amalg \times_{F(V_I)} F(E_I)$ as before as pullbacks of $F(p_I)$ and $p_B$, respectively.
Let $p' : Y' \to V_\amalg'$ and $q' : Z' \to Y'$ be pullbacks along $\varepsilon_{V_\amalg}$ of $p$ and $q$, respectively.

Since $p' \circ q'$ is a quasifibration relative to $\scat{C}'$, we can factor it into a trivial cofibration $t : Z' \to E_\amalg'$ followed by a fibration $p_\amalg' : E_\amalg' \twoheadrightarrow V_\amalg'$.
We have a map $[F(r_I),r_B] : F(\Gamma) \to V_\amalg$ as before.
This map factors as a map $r_\amalg' : F(\Gamma) \to V_\amalg'$ followed by $\varepsilon_{V_\amalg}$.
We define the interpretation of $\coprod_{i : I} B$ as the pair $r_\amalg'$, $p_\amalg'$

Now, let us describe the interpretation of the introduction rule:
\begin{center}
\AxiomC{$\Gamma \vdash j : I$}
\UnaryInfC{$\Gamma \mid x : B_j \vdash \fs{in}_j(x) : \coprod_{i : I} B_i$}
\DisplayProof
\end{center}
The interpretation of $j$ is a section $j : \Gamma \to E_I$ of $p_I$ over $r_I$.
This map determines a section $j' : \Gamma \to \Gamma.I$ of the map $\Gamma.I \to \Gamma$.
Let $\Gamma.B_j$ be the following pullback:
\[ \xymatrix{ \Gamma.B_j \ar[rr] \ar@{->>}[d]_{p_{B_j}} \pb &                           & E_B \ar@{->>}[d]^{p_B} \\
              F(\Gamma) \ar[r]_-{F(j')}                     & F(\Gamma.I) \ar[r]_-{r_B} & V_B
            } \]
The map $\Gamma.B_j \to E_B$ induces a map $b' : \Gamma.B_j \to Z'$.
We define the interpretation of $\fs{in}_j$ as $t \circ b'$.

Finally, let us describe the interpretation of the elimination rule:
\begin{center}
\AxiomC{$\Gamma \mid z : \coprod_{i : I} B_i \vdash D$}
\AxiomC{$\Gamma, i : I \mid x : B_i \vdash d : D[\fs{in}_i(x)/z]$}
\BinaryInfC{$\Gamma \mid z : \coprod_{i : I} B_i \vdash \coprod\text{-}\fs{elim}(z.D, i x.d) : D$}
\DisplayProof
\end{center}
The interpretation of $D$ consists of a map $r_D : \Gamma.\amalg \to V_D$ together with a fibration $p_D : E_D \twoheadrightarrow V_D$, where $\Gamma.\amalg$ is the pullback of $r_\amalg'$ and $p_\amalg'$.
The interpretation of $d$ is a section $d : \Gamma.I.B \to E_D$ of $p_D$, where $\Gamma.I.B$ is the pullback of $r_B$ and $p_B$.

Let $s : \Pi_{p_\amalg'}(E_\amalg' \times V_D) \times_{V_\amalg'} Z' \to \Pi_{p_\amalg'}(E_\amalg' \times V_D)$ be the pullback of $p' \circ q'$ as before.
Let $X$ be the pullback of $p_D$ over $\Pi_{p_\amalg'}(E_\amalg' \times V_D) \times_{V_\amalg'} Z'$.
Let $P = \Pi_s(X)$ and let $t'$ be the pullback of $t$:
\[ \xymatrix{ P' \times_{V_\amalg'} Z' \ar[r] \ar[d]_{t'} \pb   & P \times_{V_\amalg'} Z' \ar[r] \ar[d] \pb                 & Z' \ar[d]^t                           \\
              P' \times_{V_\amalg'} E_\amalg' \ar[r] \ar[d] \pb & P \times_{V_\amalg'} E_\amalg' \ar[r] \ar@{->>}[d] \pb    & E_\amalg' \ar@{->>}[d]^{p_\amalg'}    \\
              P' \ar[r]_{\varepsilon_P}                         & P \ar[r]                                                  & V_\amalg'
            } \]
Since $p' \circ q'$ is a quasifibration, $t'$ is a trivial cofibration.
It follows that we have a lift in the following diagram:
\[ \xymatrix{ P' \times_{V_\amalg'} Z' \ar[r] \ar[d]_{t'}               & P \times_{V_\amalg'} Z' \ar[r]        & E_D \ar@{->>}[d]^{p_D} \\
              P' \times_{V_\amalg'} E_\amalg' \ar[r] \ar@{-->}[urr]^e   & P \times_{V_\amalg'} E_\amalg' \ar[r] & V_D
            } \]
The map $F(\Gamma) \to P$ is defined as before.
It factors through a map $F(\Gamma) \to P'$.
We define the interpretation of the eliminator as the composite $\Gamma.\amalg \to P' \times_{V_\amalg} E_\amalg \xrightarrow{e} E_D$, where the first map is induced by the map $F(\Gamma) \to P'$.
\end{proof}

\subsection{Pushouts}

The construction of pushouts was described in \cite{lum-shul-hits}.
It will be convenient to split it into two parts.

\begin{lem}[pushouts]
Let $\scat{B}$ be a category with fibrations, let $\scat{C}$ be a category with fibrations which has path types (as defined in Section~\ref{sec:ext}), and let $F : \scat{B} \to \scat{C}$ be a functor.
Suppose that, for every pair of maps $f : A \to B$ and $g : A \to C$ over an object $\Gamma$ of $\scat{C}$, the map $A \times \Delta^1 \amalg_{A \amalg A} B \amalg C \to \Gamma$ is a quasifibration.
Then $F^*(\scat{C}_!)$ has stable dependent pushouts.

If $\scat{C}$ is right proper, then the map $A \times \Delta^1 \amalg_{A \amalg A} B \amalg C \to \Gamma$ is a quasifibration.
\end{lem}
\begin{proof}
The construction is essentially the same as the construction of pushouts in sections 4, 5, and 6 of \cite{lum-shul-hits}.
\end{proof}

\begin{lem}[pushouts-proper]
If $\scat{C}$ is a locally Cartesian closed right proper type-theoretic model category which has path types, then it satisfies the conditions of \rlem{pushouts}.
\end{lem}
\begin{proof}
This was proved in \cite[Theorem~4.2]{lum-shul-hits}.
The category $\scat{C}$ is assumed to be a locally Cartesian closed simplicial type-theoretic model category there, but every such category has path types and the existence of path types is enough for this proof to go through.
\end{proof}

\begin{example}
If $\scat{M}$ is the category of simplicial presheaves on a small category with its Cisinski model structure \cite{cisinski-presheaf,cisinski} and $\scat{J}$ is a small category,
then the injective model structure on $\scat{M}^\scat{J}$ satisfies the conditions of \rlem{pushouts}, so it has stable dependent pushouts.
\end{example}

Finally, let us state the unstable version of \rlem{pushouts}:

\begin{lem}[unstable-pushouts]
Let $\scat{B}$ be a category with fibrations, let $\scat{C}$ be a category with fibrations which has path types, and let $F : \scat{B} \to \scat{C}$ be a functor between them.
Suppose that $\scat{C}$ has a full subcategory $\scat{C}'$ such that, for every object $X \in \scat{C}$, there is a map $\varepsilon_X : X' \to X$ such that $X' \in \scat{C}'$ and,
for every fibrant object $\Gamma \in \scat{B}$, every map of the form $F(\Gamma) \to X$ factors through $\varepsilon_X$ and the factorization is natural in $\Gamma$.
Suppose that, for every pair of maps $f : A \to B$ and $g : A \to C$ over an object $\Gamma$ of $\scat{C}'$, the map $A \times \Delta^1 \amalg_{A \amalg A} B \amalg C \to \Gamma$ is a quasifibration relative to $\scat{C}'$.
Then $F^*(\scat{C}_!)$ has pushouts.

If trivial cofibrations of $\scat{C}$ are closed under pullbacks along pullbacks of every fibration with a codomain in $\scat{C}'$, then the map $A \times \Delta^1 \amalg_{A \amalg A} B \amalg C \to \Gamma$ is a quasifibration relative to $\scat{C}'$.
\end{lem}
\begin{proof}
The proof is the same as the proof of \rlem{pushouts} with adjustments as in \rprop{unstable-coproducts}.
\end{proof}

\section{Localization}
\label{sec:loc}

In this section, we consider the following problem.
Let $\scat{C}$ be a category and let $\Fib$ and $\Fib'$ be two classes of fibrations of $\scat{C}$ such that $\Fib' \subseteq \Fib$.
If we know that a contextually indexed contextual category of the form $F^*((\scat{C},\Fib)_!)$ has some categorical constructions, when does $F^*((\scat{C},\Fib')_!)$ also have these constructions?

\subsection{Identity types}

Let $(\scat{C},\Fib)$ be a category with fibrations and let $\Fib'$ be a subclass of $\Fib$.
We will say that $\Fib'$ is \emph{closed under identity types} if every map in $((\scat{C},\Fib')/\Gamma)_\fs{f}$ factors into a trivial cofibration of $(\scat{C},\Fib)$ followed by a fibration in $\Fib'$.

\begin{lem}[subfib]
Let $(\scat{C},\Fib)$ be a category with fibrations and let $\Fib'$ be a subclass of $\Fib$.
Then $\Fib'$ is closed under identity types if and only if the following conditions hold:
\begin{enumerate}
\item \label{it:subfib-factor} Every map in $((\scat{C},\Fib')/\Gamma)_\fs{f}$ factors into a trivial cofibration of $(\scat{C},\Fib')$ followed by a fibration in $\Fib'$.
\item \label{it:subfib-we} A map in $((\scat{C},\Fib')/\Gamma)_\fs{f}$ is a trivial cofibration of $(\scat{C},\Fib')$ if and only if it is a trivial cofibration of $(\scat{C},\Fib)$.
\end{enumerate}
\end{lem}
\begin{proof}
First, suppose that $\Fib'$ is closed under identity types.
Since every trivial cofibration of $(\scat{C},\Fib)$ is a trivial cofibration of $(\scat{C},\Fib')$, the first condition and the ``if'' part of the second condition are obvious.
Let $f$ be a trivial cofibration in $((\scat{C},\Fib')/\Gamma)_\fs{f}$.
Factor it into a map $i$ which is a trivial cofibration of $(\scat{C},\Fib)$ followed by a fibration in $\Fib'$.
Since $f$ has the left lifting property with respect to $\Fib'$, the standard argument implies that it is a retract of $i$, hence also is a trivial cofibration of $(\scat{C},\Fib)$.

To prove the converse, it is enough to factor a map into a trivial cofibration in $((\scat{C},\Fib')/\Gamma)_\fs{f}$ followed by a fibration in $\Fib'$.
Then the second condition implies that the first map is also a trivial cofibration of $(\scat{C},\Fib)$.
\end{proof}

\begin{lem}[subfib-model-cats]
Let $(\scat{C},\Fib)$ be either a type-theoretic semi-fibration category or a model category and let $\Fib'$ be a subclass of $\Fib$ which is closed under retracts, compositions, and pullbacks and contains all trivial fibrations.
Then the following conditions are equivalent:
\begin{enumerate}
\item \label{it:subfib-model-id} The class $\Fib'$ is closed under identity types.
\item \label{it:subfib-model-fib} A map in $((\scat{C},\Fib')/\Gamma)_\fs{f}$ belongs to $\Fib$ if and only if it belongs to $\Fib'$.
\item \label{it:subfib-model-path} For every fibration $p : Y \twoheadrightarrow \Gamma$ in $\Fib'$,
the diagonal $Y \to Y \times_\Gamma Y$ factors as a trivial cofibration $Y \to P(Y)$ of $\scat{C}$ followed by a fibration $P(Y) \twoheadrightarrow Y \times_\Gamma Y$ in $\Fib'$.
\end{enumerate}
\end{lem}
\begin{proof}
\eqref{it:subfib-model-id} $\implies$ \eqref{it:subfib-model-fib}
Since $\Fib' \subseteq \Fib$, we just need to prove that every map in $((\scat{C},\Fib')/\Gamma)_\fs{f}$ which belongs to $\Fib$ also belongs to $\Fib'$.
Let $f : X \to Y$ be such a map.
Factor it into a trivial cofibration $i : X \to Z$ followed by a fibration $p : Z \twoheadrightarrow Y$ in $\Fib'$.
Since $i$ has the left lifting property with respect to $f$, the standard argument implies that $f$ is a retract of $p$.
Since $\Fib'$ is closed under retracts, $f$ belongs to $\Fib'$.

\eqref{it:subfib-model-fib} $\implies$ \eqref{it:subfib-model-path}
Let $p : Y \twoheadrightarrow \Gamma$ be a fibration in $\Fib'$.
Factor the diagonal $Y \to Y \times_\Gamma Y$ into a trivial cofibration $Y \to P(Y)$ of $\scat{C}$ followed by a fibration $q : P(Y) \to Y \times_\Gamma Y$ in $\Fib$.
Since $\Fib'$ is closed under pullbacks, the projection $\pi_1 : Y \times_\Gamma Y \to Y$ belongs to $\Fib'$.
Since $\Fib'$ contains all trivial fibrations, the map $\pi_1 \circ q$ also belongs to $\Fib'$.
It follows that $q$ belongs to $\Fib'$.

\eqref{it:subfib-model-path} $\implies$ \eqref{it:subfib-model-id}
Let $f : X \to Y$ be a map in $((\scat{C},\Fib')/\Gamma)_\fs{f}$.
We can factor $f$ as usual as $X \to Z \xrightarrow{q} X \times_\Gamma Y \xrightarrow{\pi_2} Y$, where $q$ is a pullback of $P(Y) \twoheadrightarrow Y \times_\Gamma Y$:
\[ \xymatrix{ Z \ar[r] \ar@{->>}[d]_q \pb                       & P(Y) \ar@{->>}[d] \\
              X \times_\Gamma Y \ar[r] \ar@{->>}[d]_{\pi_1} \pb & Y \times_\Gamma Y \ar@{->>}[d]^{\pi_1} \\
              X \ar[r]_f                                        & Y
            } \]
Since $\Fib'$ is closed under pullbacks, $q$ belongs to $\Fib'$.
The projection $\pi_2 : X \times_\Gamma Y \to Y$ also belongs to $\Fib'$ since it is a pullback of the map $X \twoheadrightarrow \Gamma$ which belongs to $\Fib'$.
Since $\Fib'$ is closed under compositions, the map $\pi_2 \circ q$ belongs to $\Fib'$.
Thus, we just need to prove that the map $X \to Z$ factors into a trivial cofibration followed by a fibration in $\Fib'$.
Since $\pi_1 \circ q$ is a pullback of the trivial fibration $P(Y) \twoheadrightarrow Y$, it is also a trivial fibration.
Thus, $X \to Z$ is a weak equivalence by 2-out-of-3.
Factor it into a trivial cofibration $X \to Z'$ followed by a fibration $Z' \to Z$.
By 2-out-of-3, $Z' \to Z$ is a trivial fibration.
Hence, it belongs to $\Fib'$.
\end{proof}

\begin{prop}[localized-fib]
Let $(\scat{C},\Fib)$ be a type-theoretic semi-fibration category and let $\Fib'$ be a subclass of $\Fib$ which contains identity morphisms and closed under compositions, pullbacks, and identity types.
Then $(\scat{C},\Fib')$ is also a type-theoretic semi-fibration category.
Moreover, $(\scat{C},\Fib')$ is closed under identity types in the sense that identity types in this category are equivalent to identity types in $(\scat{C},\Fib)$.
\end{prop}
\begin{proof}
The fact that $(\scat{C},\Fib')$ is a type-theoretic semi-fibration category easily follows from \rlem{subfib}.
The interpretation of identity types is defined as a factorization of the map $E_A \to E_A \times_{V_A} E_A$ into a trivial cofibration followed by a fibration.
We can factor this map into a trivial cofibration of $(\scat{C},\Fib)$ followed by a fibration in $\Fib'$.
Then this type is an interpretation of the identity type in both $(\scat{C},\Fib)$ and $(\scat{C},\Fib')$.
Since identity types are unique up to an equivalence, this implies that they are equivalent in these categories.
\end{proof}

\begin{cor}[localized-con-fib]
Let $F$ be a left relative Quillen functor between type-theoretic semi-fibration categories $\scat{B}$ and $\scat{C}$ such that $F^*(\scat{C}_!)$ has extensional identity types.
Let $\Fib'$ be a subclass of fibrations of $\scat{C}$ which contains identity morphisms and closed under compositions, pullbacks, and identity types.
Then $F^*((\scat{C},\Fib')_!)$ is a contextually indexed contextual category with unit types, $\Sigma$-types, and extensional identity types.
\end{cor}
\begin{proof}
This follows from \rprop{indexed-locally-small} and \rprop{localized-fib}.
\end{proof}

The following theorem provides a large class of examples of contextually indexed contextual categories.

\begin{thm}
Let $\scat{C}$ be a right proper combinatorial type-theoretic semi-fibration category in which all objects are cofibrant and let $F : \scat{J} \to \scat{I}$ be a functor between small categories.
Let $\Fib$ be a class of injective fibrations which contains identity morphisms and closed under compositions, pullbacks, and identity types.
Then $F^*((\scat{C}^\scat{J},\Fib)_!)$ is a locally small contextually indexed contextual category with dependent $\Hom$-types, unit types, $\Sigma$-types, and extensional identity types.
\end{thm}
\begin{proof}
By \rcor{injective-combinatorial} and \rprop{indexed-locally-small}, $(\scat{C}^\scat{J},\Fib_\fs{inj})$ is a type-theoretic semi-fibration category with dependent $\Hom$-types, unit types, $\Sigma$-types, and extensional identity types.
By \rcor{localized-con-fib}, $F^*((\scat{C}^\scat{J},\Fib)_!)$ is a contextually indexed contextual category with unit types, $\Sigma$-types, and extensional identity types.
By \rprop{localized-hom}, it has dependent $\Hom$-types.
\end{proof}

\subsection{Localization of model categories}

Let $S$ be a class of maps of a model category $\scat{C}$.
Then the class of fibrations that have the right lifting property with respect to $S$ is not necessarily closed under identity types, but if $S$ satisfies certain closure conditions, then this is true.
If $\Gamma$ is an object of $\scat{C}$, we will write $S/\Gamma$ for the class of maps $\scat{C}/\Gamma$ which consists of maps such that their underlying maps in $\scat{C}$ belong to $S$.
We will say that $S$ is closed under \emph{coidentity types} if, for every object $\Gamma$ and every map $i : U \to V$ in $S/\Gamma$, there exist cylinder objects $C(U)$ for $U$ and $C(V)$ for $V$
and a map of cylinder objects $C(i) : C(U) \to C(V)$ such that the map $i' : C(U) \amalg_{(U \amalg U)} (V \amalg V) \to C(V)$
has the left lifting property with respect to every fibrant object of $\scat{C}/\Gamma$ which has the right lifting property with respect to $S$ as a map in $\scat{C}$.

\begin{remark}[coidentity]
Let $S$ be a class of maps of a category $\scat{C}$.
We will write $\Iinj[S]$ for the class of maps of $\scat{C}$ that have the right lifting property with respect to $S$.
We will write $\Icof[S]$ for the class of maps of $\scat{C}$ that have the left lifting property with respect to $\Iinj[S]$.

If $\J$ is the class of trivial cofibrations of $\scat{C}$ and $S$ is a class of maps of $\scat{C}$ such that, for every $i \in S$, there exists a map $i'$ as described above which belongs to $\Icof[(S \cup \J)]$, then $S$ is closed under coidentity types.
\end{remark}

\begin{lem}[coidentity]
Let $S$ be a class of cofibrations of a model category $\scat{C}$ closed under coidentity types.
Suppose that either the domains of maps in $S$ are cofibrant or $\scat{C}$ is left proper.
Then the class of fibrations that have the right lifting property with respect to $S$ is closed under identity types.
\end{lem}
\begin{proof}
Let $P(Y)$ be a path object for a fibration $Y \twoheadrightarrow \Gamma$ which has the right lifting property with respect to $S$.
By \rlem{subfib-model-cats}, we just need to prove that the fibration $P(Y) \twoheadrightarrow Y \times_\Gamma Y$ also has this property.
Suppose that we have a commutative square as below, where $i \in S$.
\[ \xymatrix{ U \ar[r]^-h \ar[d]_i                  & P(Y) \ar@{->>}[d] \\
              V \ar[r]_-{\langle y_0, y_1 \rangle}  & Y \times_\Gamma Y
            } \]
We need to construct a lift in this square.

Let $U \amalg U \xrightarrow{[i_0^U,i_1^U]} C(U) \xrightarrow{s^U} U$ and $V \amalg V \xrightarrow{[i_0^V,i_1^V]} C(V) \xrightarrow{s^V} V$ be cylinder objects,
let $W = C(U) \amalg_{(U \amalg U)} (V \amalg V)$, and let $i' : W \to C(V)$ be the map described in the definition of the closure under coidentity types.
Consider the following diagram:
\[ \xymatrix{                            & U \ar[d] \ar[r]^i        & V \ar[d]                                             &                                            &                \\
                                         & U \amalg U \ar[r] \ar[d] & \po V \amalg U \ar[r]^-{[t_0, h]} \ar[d]             & P(Y) \ar@{->>}[r]^{p_1} \ar@{->>}[d]^{p_0} & Y \ar@{->>}[d] \\
              U \ar[r]^-{i^U_1} \ar[d]_i & C(U) \ar[r]              & \po C(U) \amalg_U V \ar[r] \ar[d] \ar@{-->}[ur]^{h'} & Y \ar@{->>}[r]                             & \Gamma         \\
              V \ar[rr]                  &                          & \po W \ar[r]_-{i'} & C(V) \ar[ur]                    &                                            &
            } \]
where the vertical map $U \to C(U)$ is $i_0^U$, the map $C(V) \to \Gamma$ is the composite $C(V) \xrightarrow{s^V} V \to \Gamma$,
the map $C(U) \to Y$ is the composite $C(U) \xrightarrow{s^U} U \xrightarrow{i} V \xrightarrow{y_0} Y$, and $t_0 : V \to P(Y)$ is the composite $V \xrightarrow{y_0} Y \to P(Y)$.
Since $V \amalg U \to C(U) \amalg_U V$ is a cofibration and $p_0 : P(Y) \twoheadrightarrow Y$ is a trivial fibration, we have a lift $h' : C(U) \amalg_U V \to P(Y)$.
The map $p_1 \circ h'$ extends to a map $h'' : W \to Y$ such that its composition with $V \to W$ equals to $y_1$:
\[ \xymatrix{                                      & P(Y) \ar@{->>}[r]^{p_1}            & Y \ar@{->>}[d] \\
              C(U) \amalg_U V \ar[d] \ar[ur]^{h'}  &                                    & \Gamma         \\
              W \ar[r]_-{i'} \ar@{-->}[uurr]^{h''} & C(V) \ar[ur] \ar@{-->}[uur]^{h'''} &
            } \]
Since $Y \twoheadrightarrow \Gamma$ is a fibration which has the right lifting property with respect to $S$, we have a lift $h''' : C(V) \to Y$ by the definition of $i'$.
Now, consider the following diagram:
\[ \xymatrix{ U \ar[r]^{i_1^U} \ar[d]_i & C(U) \ar[r]   & C(U) \amalg_U V \ar[r]^-{h'} \ar[d]_j                            & P(Y) \ar@{->>}[d] \\
              V \ar[r]                  & W \ar[r]_{i'} & C(V) \ar[r]_-{\langle y_0 \circ s, h''' \rangle} \ar@{-->}[ur]^q & Y \times_\Gamma Y
            } \]
The map $j$ is the composite of $C(U) \amalg_U V \to W$ and $i' : W \to C(V)$.
Both of these maps are cofibrations.
Moreover, $j$ is a weak equivalence.
To prove this, it is enough to prove that the map $V \to C(U) \amalg_U V$ is a weak equivalence since the map $i_0^V : V \to C(V)$ is a weak equivalence.
The map $V \to C(U) \amalg_U V$ is the pushout of $U \to C(U)$ along $i : U \to V$.
If $U$ is cofibrant, then $i_0^U : U \to C(U)$ is a trivial cofibration, so $V \to C(U) \amalg_U V$ is also a trivial cofibration.
If $\scat{C}$ is left proper, then $V \to C(U) \amalg_U V$ is a weak equivalence since it is a pushout of a weak equivalence along a cofibration.

Since $j$ is a trivial cofibration and $P(Y) \to Y \times_\Gamma Y$ is a fibration, we have a lift $q : C(V) \to P(Y)$ in the diagram above.
The map $V \to W \xrightarrow{i'} C(V) \xrightarrow{q} P(Y)$ is a lift in the original square.
\end{proof}

A \emph{cosimplicial resolution} of an object $X$ of $\scat{C}$ is a cofibrant replacement of the constant functor on $X$ in the Reedy model structure on $\scat{C}^\Delta$.
Let $\Delta$ be a fibrant cofibrant replacement of the terminal object of $\scat{C}^\Delta$.
That is, $\Delta$ is a functor $\Delta \to \scat{C}$ such that $\Delta^n \to 1$ is a trivial fibration for all $n$ and $\partial \Delta^n \to \Delta^n$ is a cofibration in $\scat{C}$.
Then $n \mapsto Q(X) \times \Delta^n$ is a cosimplicial resolution of $X$, where $Q(X)$ is any cofibrant replacement of $X$.

A \emph{cofibrant cosimplicial resolution} of a map $f : X \to Y$ consists of cosimplicial resolution of its domain and codomain and a cofibration between them such that the obvious square commutes.
For a class $S$ of maps of a model category, a class $\Lambda(S)$ of \emph{horns on $S$} \cite[Definition~3.3.8]{hirschhorn} is defined as the class of maps of the form
\[ X \otimes \Delta^n \amalg_{X \otimes \partial \Delta^n} Y \otimes \partial \Delta^n \to Y \otimes \Delta^n, \]
where $X \otimes \Delta^n \to Y \otimes \Delta^n$ is a cofibrant cosimplicial resolution of a map $X \to Y$ from $S$.

\begin{lem}[coidentity-lambda]
Let $S$ be a class of maps of a model category $\scat{C}$.
Then any class of horns $\Lambda(S)$ on $S$ is closed under coidentity types and consists of cofibrations between cofibrant objects.
\end{lem}
\begin{proof}
The class $\Lambda(S)$ consists of cofibrations between cofibrant objects by \cite[Corollary~16.3.11]{hirschhorn}.
Let $i : Q(U) \otimes \Delta^n \amalg_{Q(U) \otimes \partial \Delta^n} Q(V) \otimes \partial \Delta^n \to Q(V) \otimes \Delta^n$ be a map in $\Lambda(S)$.
Note that $(X \otimes K) \amalg (X \otimes K) \simeq X \otimes (K \times \partial \Delta^1)$ and that $X \otimes (K \times \Delta^1)$ is a cylinder object for $X \otimes K$.
The second statement follows from the first one and \cite[Lemma~16.4.4]{hirschhorn}.
This lemma and \cite[Proposition~16.5.6]{hirschhorn} imply that maps $Q(U) \otimes K \to Q(U) \otimes L$ and $Q(U) \otimes K \to Q(V) \otimes K$ are cofibrations between cofibrant objects for all cofibrations $K \to L$.
This implies that $Q(U) \otimes (\Delta^n \times \Delta^1) \amalg_{Q(U) \otimes (\partial \Delta^n \times \Delta^1)} Q(V) \otimes (\partial \Delta^n \times \Delta^1)$ is a cylinder object
for $Q(U) \otimes \Delta^n \amalg_{Q(U) \otimes \partial \Delta^n} Q(V) \otimes \partial \Delta^n$.

By \rremark{coidentity}, it is enough to prove that $i'$ belongs to $\Icof[S]$.
By \cite[Proposition~16.4.3]{hirschhorn}, $X \otimes -$ preserves colimits.
It follows that $i'$ is isomorphic to the map $Q(U) \otimes (\Delta^n \times \Delta^1) \amalg_{(Q(U) \otimes \partial(\Delta^n \times \Delta^1))} Q(V) \otimes \partial(\Delta^n \times \Delta^1) \to Q(V) \otimes (\Delta^n \times \Delta^1)$,
where $\partial(\Delta^n \times \Delta^1) = \partial \Delta^n \times \Delta^1 \amalg_{(\partial \Delta^n \times \partial \Delta^1)} \Delta^n \times \partial \Delta^1$.
Since $\partial(\Delta^n \times \Delta^1) \to \Delta^n \times \Delta^1$ is a cofibration, \cite[Proposition~16.4.5]{hirschhorn} implies that $i'$ belongs to $\Icof[S]$.
\end{proof}

\begin{remark}
Let $K$ be a class of objects of a type-theoretic model category $\scat{C}$.
A map $X \to Y$ of $\scat{C}$ is \emph{$K$-local equivalence} if the induced map of homotopy function complexes $\fs{map}(Y,W) \to \fs{map}(X,W)$ is a weak equivalence for every $W \in K$.
Let $\scat{C}'$ be the category with the same class of cofibrations as $\scat{C}$ and $K$-local equivalences as weak equivalences.
In general, $\scat{C}'$ is not a model category, but it is always a type-theoretic semi-fibration category.
Indeed, by \cite[Proposition~17.8.5]{hirschhorn} and \cite[Proposition~17.8.14]{hirschhorn}, the class of $K$-local equivalences is closed under function $\Lambda(-)$.
Now, \rprop{localized-fib}, \rlem{coidentity}, and \rlem{coidentity-lambda} imply that $\scat{C}'$ is a type-theoretic semi-fibration category.
\end{remark}

\begin{example}[cisinski]
If $F : \scat{B} \to \scat{C}$ is a left relative Quillen functor and $\scat{C}$ is a Cisinski model category, then 
the contextually indexed contextual category $F^*(\scat{C}_!)$ has dependent $\Hom$-types, unit types, $\Sigma$-types, and extensional identity types.
This example was discussed in \rexample{cisinski-simp}.
The only new part is that $F^*(\scat{C}_!)$ has extensional identity types which follows from the previous remark.
\end{example}

\section{Quasicategories}
\label{sec:quasicats}

In this section, we describe a contextually indexed contextual category based on quasicategories.

Consider the functor $\Id : \sSet_K \to \sSet_J$, where $\sSet_K$ is equipped with the Kan model structure and $\sSet_J$ is equipped with the Joyal model structure.
This functor is not a left Quillen functor, but it is a relative left Quillen functor.
To construct its relative right adjoint, we define an auxiliary notion.
Let $x : \Delta^1 \to X$ be a 1-simplex of a simplicial set $X$.
We will call such a simplex an \emph{edge} of $X$.
We will say that $x$ is invertible if it extends to a map $\bDelta^1 \to X$, where $\bDelta^1$ is the nerve of the groupoid with two objects and a unique arrow between any pair of objects.

For every simplicial set $X$, let $G(X)$ be the subset of $X$ consisting of simplices in which every edge is invertible.
We will denote the inclusion map $G(X) \to X$ by $\varepsilon_X$.
Clearly, $G : \sSet \to \sSet$ is a functor.
Moreover, it is a right adjoint to $\Id$ relative to Kan complexes.
Since $\varepsilon_X$ is a monomorphism, to prove this, it is enough to show that any map $K \to X$ factors through $\varepsilon_X$ if $K$ is a Kan complex.
Every edge $\Delta^1 \to K$ extends to a map $\bDelta^1 \to K$ since $\Delta^1 \to \bDelta^1$ is a trivial cofibration in the Kan model structure.
It follows that every simplex of the form $\Delta^n \to K \to X$ belongs to $G(X)$, that is $K \to X$ factors through $\varepsilon_X$.

We already saw that every edge of a Kan complex is invertible.
The following lemma shows that this is also true for objects of the form $G(Y)$:

\begin{lem}
For every simplicial set $X$, the following conditions are equivalent:
\begin{enumerate}
\item Every edge of $X$ is invertible.
\item $X$ has the right lifting property with respect to the map $\Delta^1 \to \bDelta^1$.
\item $G(X) = X$.
\item $X$ is isomorphic to $G(Y)$ for some simplicial set $Y$.
\end{enumerate}
\end{lem}
\begin{proof}
The equivalence of the first two condition follows immediately from the definition of an invertible edge.
If every edge of $X$ is invertible, then every simplex of $X$ belongs to $G(X)$, so $G(X) = X$.
Clearly, $G(X) = X$ implies that $X$ is of the form $G(Y)$.
Thus, we just need to prove that every edge of $G(Y)$ is invertible.
Indeed, by definition of $G(Y)$, for every simplex $u : \Delta^1 \to G(Y)$, there is a map $v : \bDelta^1 \to Y$ such that the following square commutes:
\[ \xymatrix{ \Delta^1 \ar[r]^-u \ar[d]         & G(Y) \ar[d]^{\varepsilon_Y} \\
              \bDelta^1 \ar[r]_-v \ar@{-->}[ur] & Y
            } \]
Since $\bDelta^1$ is a Kan complex, $v$ factors through a map $\bDelta^1 \to G(Y)$, that is there is a diagonal arrow in the diagram above such that the bottom triangle commutes.
Since $\varepsilon_Y$ is a monomorphism, the upper triangle also commutes.
\end{proof}

Simplicial sets in which all edges are invertible often can replace Kan complexes.
For example, every categorical fibration with a Kan complex as a codomain is a Cartesian fibration.
We will prove in \rprop{G-Cart} that this is also true for categorical fibrations with a codomain in which all edges are invertible.
The general theory of localizations of model categories implies that a map between Kan complexes factors into a categorical trivial cofibration followed by a Kan fibration.
We will prove in \rprop{G-Kan} that this is also true for maps between simplicial sets in which all edges are invertible.
To prove these propositions, we need to define an auxiliary construction.

Let $S$ be the set of non-degenerate edges of a simplicial set $X$.
Then we define a simplicial set $b(X)$ as the following pushout:
\[ \xymatrix{ \coprod\limits_S \Delta^1 \ar[r] \ar[d]  & X \ar[d] \\
              \coprod\limits_S \bDelta^1 \ar[r]        & \po b(X)
            } \]
Note that $b(\Delta^1) = \bDelta^1$.
It is easy to see that $b : \sSet \to \sSet$ is a functor which preserves cofibrations.

\begin{lem}[b-cof]
Every map of the form $b(\Lambda^n_k) \to b(\Delta^n)$ is a categorical trivial cofibration.
\end{lem}
\begin{proof}
This map is a cofibration since $b$ preserves cofibrations.
To prove that it is a categorical equivalence, it is enough to show that if $X$ is a contractible simplicial set, then $b(X)$ is categorically equivalent to $\Delta^0$.
Let $R(b(X))$ be the fibrant replacement of $b(X)$ in the Joyal model structure.
Since $R(b(X))$ is contractible, it is enough to show that it is a Kan complex.
By \cite[Proposition~1.2.5.1]{lurie-topos}, to prove that $R(b(X))$ is a Kan complex, it is enough to show that every edge of $R(b(X))$ is an equivalence.
This follows from the fact that every edge of $R(b(X))$ is a composition of edges that belong to $b(X)$ and every edge that belongs to $b(X)$ is an equivalence by the definition of $b$.
\end{proof}

\begin{prop}[G-Cart]
If every edge of a simplicial set $Y$ is invertible, then every categorical fibration $p : X \to Y$ is a Cartesian fibration.
\end{prop}
\begin{proof}
Let $\Delta^1 \to Y$ be an edge of $Y$.
This map factors through a map $\bDelta^1 \to Y$.
Since $\bDelta^1$ is a Kan complex, the pullback of $p : X \to Y$ along $\bDelta^1 \to Y$ is a Cartesian fibration by \cite[Proposition~.3.3.1.8]{lurie-topos}.
It follows that $p : X \to Y$ is a locally Cartesian fibration.

Let us prove that an edge of $X$ is locally $p$-Cartesian if and only if it is invertible.
Let $e : \Delta^1 \to X$ be an edge of $X$.
The map $p \circ e$ factors through the map $\Delta^1 \to \bDelta^1$.
Let $p' : X' \to \bDelta^1$ be the pullback of $p$ along the map $\bDelta^1 \to Y$ and let $e' : \Delta^1 \to X'$ be the edge induced by $e$.
By \cite[Remark~2.4.1.12]{lurie-topos}, $e$ is a locally $p$-Cartesian if and only if $e'$ is locally $p'$-Cartesian.
By \cite[Proposition~2.4.2.8]{lurie-topos}, $e'$ is locally $p'$-Cartesian if and only if it is $p'$-Cartesian since $p'$ is a Cartesian fibration.
By \cite[Proposition~2.4.1.5]{lurie-topos}, $e'$ is $p'$-Cartesian if and only if it is invertible.
Thus, if $e$ is locally $p$-Cartesian, then $e'$ is invertible and hence $e$ is also invertible.
Conversely, suppose that $e$ is invertible.
Then $e$ factors through the map $\Delta^1 \to \bDelta^1$.
Let $f : \bDelta^1 \to X$ be the map through which it factors.
Let $p' : X' \to \bDelta^1$ be the pullback of $p$ along $p \circ f$.
Then $f$ induces a section $s : \bDelta^1 \to X'$ of $p'$.
It follows that $e'$ factors through $s$ and hence it is invertible.
Thus, $e$ is locally $p$-Cartesian.

By \cite[Proposition~2.4.2.8]{lurie-topos}, $p : X \to Y$ is a Cartesian fibration if locally $p$-Cartesian edges are closed under composition.
Let $s : \Delta^2 \to X$ be a simplex such that edges induced by $\Delta^{\{0,1\}} \to \Delta^2$ and $\Delta^{\{1,2\}} \to \Delta^2$ are invertible.
We need to show that the third edge is also invertible.
The map $s$ factors through the map $\Delta^2 \to \Delta^2 \amalg_{\Lambda^2_1} b(\Lambda^2_1)$.
Let $t : \Delta^2 \amalg_{\Lambda^2_1} b(\Lambda^2_1) \to X$ be the map through which it factors.
Since edges of $Y$ are invertible, the map $p \circ t$ factors through the map $i : \Delta^2 \amalg_{\Lambda^2_1} b(\Lambda^2_1) \to b(\Delta^2)$.
Thus, we have the following commutative square:
\[ \xymatrix{ \Delta^2 \amalg_{\Lambda^2_1} b(\Lambda^2_1) \ar[r]^-t \ar[d]_i   & X \ar[d]^p \\
              b(\Delta^2) \ar[r]                                                & Y
            } \]

To finish the proof, we need to show that this square has a lift.
Since $p$ is a categorical fibration and $i$ is a cofibration, we just need to prove that $i$ is a categorical equivalence.
Consider the sequence of arrows $b(\Lambda^2_1) \to \Delta^2 \amalg_{\Lambda^2_1} b(\Lambda^2_1) \xrightarrow{i} b(\Delta^2)$.
The first arrow is the pushout of the map $\Lambda^2_1 \to \Delta^2$ and hence it is a categorical equivalence.
The composite is a categorical equivalence by \rlem{b-cof}.
Thus, $i$ is a categorical equivalence by 2-out-of-3.
\end{proof}

\begin{prop}[G-Kan]
If edges of simplicial sets $X$ and $Y$ are invertible, then every map $X \to Y$ factors into a categorical trivial cofibration $X \to Z$ followed by a Kan fibration $Z \to Y$.
\end{prop}
\begin{proof}
Let $J$ be the set of maps of the form $b(\Lambda^n_k) \to b(\Delta^n)$ for all $0 \leq k \leq n$.
Let $f : X \to Y$ be a map.
We can apply the small object argument to the set $J$ to factor $f$ into a map $X \to Z$ in $\Icell[J]$ followed by a map $Z \to Y$ in $\Iinj[J]$.
The class $\Icell[J]$ consists of relative $J$-cell complexes, that is transfinite compositions of pushouts of maps in $J$.

Let us prove that every edge of $Z$ is invertible.
By definition, $Z = X_\lambda$ for some ordinal $\lambda$.
We prove by induction on $\alpha$ that every edge in $X_\alpha$ is invertible.
Every edge in $X_0 = X$ is invertible by assumption.
If $\alpha$ is a limit ordinal, then every edge of $X_\alpha$ belongs to $X_\beta$ for some $\beta < \alpha$.
Thus, it is invertible by the induction hypothesis.
If $\alpha = \alpha' + 1$, then every edge of $X_\alpha$ belongs either to $X_{\alpha'}$ or to $b(\Delta^n)$ and it is invertible in both of these cases.

By \rlem{b-cof}, the map $X \to Z$ is a categorical trivial cofibrations.
Now, let us prove that the map $Z \to Y$ is a Kan fibration.
Consider a commutative square of the form
\[ \xymatrix{ \Lambda^n_k \ar[r] \ar[d] & Z \ar[d] \\
              \Delta^n \ar[r]           & Y
            } \]
Since every edge of $Z$ is invertible, the map $\Lambda^n_k \to Z$ factors through the map $\Lambda^n_k \to b(\Lambda^n_k)$.
Let $D$ be the pushouts of $\Lambda^n_k \to b(\Lambda^n_k)$ and $\Lambda^n_k \to \Delta^n$.
The map $D \to Y$ factors through the map $D \to b(\Delta^n)$ since every edge of $Y$ is invertible.
Thus, we can define a lift as follows:
\[ \xymatrix{ \Lambda^n_k \ar[r] \ar[d] & b(\Lambda^n_k) \ar[rr] \ar[d] &                                   & Z \ar[d] \\
              \Delta^n \ar[r]           & \po D \ar[r]                  & b(\Delta^n) \ar@{-->}[ur] \ar[r]  & Y
            } \]
\end{proof}

Now, we can prove that $G$ is a right relative Quillen functor:

\begin{lem}[G-fib]
The functor $G$ maps categorical fibrations to Kan fibrations.
\end{lem}
\begin{proof}
Let $f : X \to Y$ be a categorical fibration.
By \rprop{G-Kan}, the map $G(f) : G(X) \to G(Y)$ factors into a categorical trivial cofibration $i : G(X) \to Z$ followed by a Kan fibration $p : Z \to G(Y)$.
Since $f$ is a categorical fibration, we have a lift in the following diagram:
\[ \xymatrix{ G(X) \ar[r]^-{\varepsilon_X} \ar[d]_i & X \ar[dd]^f   \\
              Z \ar[d]_p \ar@{-->}[ur]^j            &               \\
              G(Y) \ar[r]_-{\varepsilon_Y}          & Y
            } \]
Since $p$ is a Kan fibration and every edge of $G(Y)$ is invertible, this is also true for $Z$.
It follows that the map $j : Z \to X$ factors through $\varepsilon_X$.
Let $k : Z \to G(X)$ be the map through which it factors.
Since $\varepsilon_X \circ k \circ i = j \circ i = \varepsilon_X$ and $\varepsilon_X$ is a monomorphism, $k \circ i = \id_{G(X)}$.
Also $\varepsilon_Y \circ p \circ i \circ k = f \circ j \circ i \circ k = f \circ \varepsilon_X \circ k = f \circ j = \varepsilon_Y \circ p$.
Since $\varepsilon_Y$ is a monomorphism, $p \circ i \circ k = p$.
Thus, $G(f) = p \circ i$ is a retract of $p$:
\[ \xymatrix{ G(X) \ar[r]^i \ar[dr]_{G(f)}  & Z \ar[d]^p \ar[r]^k   & G(X) \ar[ld]^{G(f)} \\
                                            & G(Y)                  &
            } \]
Since $p$ is a Kan fibration, $G(f)$ is also a Kan fibration.
\end{proof}

Now, we are ready to prove that $\Id^*((\sSet_J)_!)$ has various categorical constructions:

\begin{thm}
Let $\Id : \sSet_K \to \sSet_J$ be the identity functor, where $\sSet_K$ is equipped with the Kan model structure and $\sSet_J$ is equipped with the Joyal model structure.
Then the contextually indexed contextual category $\Id^*((\sSet_J)_!)$ is locally small and Cartesian closed.
Moreover, it has $\Pi$-types in the empty context and dependent $\Hom$-types.
It also has $\Sigma$-types, unit types, extensional identity types, dependent products, stable dependent coproducts, strict initial types, stable binary coproducts, unstable pushouts, and (indexed) extension types with respect to cofibrations.
\end{thm}
\begin{proof}
First, let us show that $\Id^*((\sSet_J)_!)$ has $\Pi$-types in the empty context.
This follows from \rlem{cartesian-closed}.
We can take $\scat{C}'$ to be the subcategory of $\sSet$ consisting of simplicial sets with invertible edges.
Let $\varepsilon_X$ be the map $\varepsilon_X : G(X) \to X$.
The last condition follows from the fact that $G$ is a right adjoint to $\Id$ relative to Kan complexes.
Thus, we just need to prove that $\Pi_f$ maps categorical fibrations over $A$ to categorical fibrations over $B$ for every categorical fibration $f : A \to B$ such that $B$ has invertible edges.
This is true if and only if categorical equivalences are closed under pullbacks along pullbacks of $f$.
By \rprop{G-Cart}, pullbacks of $f$ are Cartesian fibrations.
By \cite[Proposition~3.3.1.3]{lurie-topos}, categorical equivalences are closed under pullbacks along Cartesian fibrations.

The fact that $\Id^*((\sSet_J)_!)$ has dependent $\Hom$-types, $\Sigma$-types, unit types, and extensional identity types follows from \rexample{cisinski}.
The existence of dependent products follows from \rprop{products}.
Indeed, $\Pi_g$ maps categorical fibrations over $A$ to categorical fibrations over $B$ for every Kan fibration $g : A \to B$ since every Kan fibration is a Cartesian fibration.

The existence of strict initial types and stable binary coproducts follows from \rprop{initial} and \rprop{binary-coproducts} since the initial type is fibrant and fibrations are closed under coproducts in $\sSet_J$.
The existence of stable dependent coproducts follows immediately from \rprop{coproducts}.
The existence of unstable pushouts follows from \rlem{unstable-pushouts}.
Finally, (indexed) extension types with respect to cofibrations exist by \rprop{extension-types} and \rprop{indexed-extension-types}.
\end{proof}

\bibliographystyle{amsplain}
\bibliography{ref}

\end{document}